\documentclass[12pt]{article}
\usepackage[margin=1in]{geometry}
\usepackage{graphicx}
\usepackage{multirow}%
\usepackage{amsmath,amssymb,amsfonts,amsthm}
\usepackage{hyperref}
\usepackage{cleveref}
\usepackage{mathrsfs}%
\usepackage[title]{appendix}%
\usepackage{xcolor}%
\usepackage{textcomp}%
\usepackage{manyfoot}%
\usepackage{booktabs}%
\usepackage{algorithm}%
\usepackage{algorithmicx}%
\usepackage{algpseudocode}%
\usepackage{listings}%
\usepackage{fixdif}

\usepackage{longtable}
\usepackage{multirow}
\usepackage{lscape}
\usepackage{subcaption}
\usepackage[width=.75\textwidth]{caption}
\usepackage{bm}
\usepackage{bbold}
\usepackage{cases}
\usepackage{authblk}
\allowdisplaybreaks

\newtheorem{theorem}{Theorem}
\newtheorem{proposition}[theorem]{Proposition}%
\newtheorem{lemma}[theorem]{Lemma}%

\theoremstyle{thmstylethree}%
\newtheorem{definition}{Definition}%

\newcommand{\mtc}[1] {\mathcal{#1}}
\newcommand{\pr}[1]{\ensuremath{\mathsf{#1}}}
\newcommand{\rv}[1]{\ensuremath{\bm{#1}}}
\def\R{\ensuremath{\mathbb{R}}}
\def\Ex{\ensuremath{\mathbb{E}}}

\def\Ind{\ensuremath{\mathbb{1}}}
\DeclareMathOperator*{\argmax}{argmax}

\begin{document}

\title{The L-Shaped Method for Stochastic Programs with Decision-Dependent Uncertainty}

\author[1]{Giovanni Pantuso}
\author[2]{Mike Hewitt}

\affil[1]{Department of Mathematical Sciences, University of Copenhagen, Denmark \\
  \href{mailto:gp@math.ku.dk}{gp@math.ku.dk}}
\affil[2]{Department of Information Systems and Supply Chain Management, Loyola University Chicago, USA \\
  \href{mailto:mhewitt3@luc.edu}{mhewitt3@luc.edu}  
}

\date{}

\maketitle

\begin{abstract}
In this paper we extend the well-known L-Shaped method to solve two-stage stochastic programming problems with decision-dependent uncertainty. 
The method is based on a novel, unifying, formulation and on distribution-specific optimality and feasibility cuts for both linear and integer stochastic programs. Extensive tests on three production planning problems illustrate that the method is extremely effective on large-scale instances.
\end{abstract}

\section{Introduction}\label{sec1}

A two-stage stochastic program can be concisely expressed as
\begin{subequations}\label{eq:sp}
\begin{align}
     \min_{x\in\mtc{X}} ~&c^\top x+\Ex_{\pr{P}}[Q(x,\rv{\xi})]
\end{align}
\noindent
where $\mtc{X}$ is a subset of $\R^{n_1}$ (potentially finite and countable), $c\in \R^{n_1}$ a known vector, and $\rv{\xi}:\Omega\to\R^N$ a random variable defined on a probability space $(\Omega,\Sigma,\pr{\mu})$ with distribution $\pr{P}$ induced by $\mu$. Furthermore, given $x$, $\Ex_{\pr{P}}[Q(x,\xi)]=\int_{\Xi}Q(x,\xi)\d\pr{P}(\xi)$, where $\Xi$ is the support of $\pr{P}$, and
\begin{equation}\label{eq:sp2s}
Q(x,\xi)=\min_{y\in\mtc{Y}}\{q^\top~ y|Wy=h-Tx\}.
\end{equation}
\end{subequations}
represents the cost of second stage decisions $y$ from $\mtc{Y}\subseteq\R^{n_2}$ given a first stage decision $x$ and a realization $\xi$ of random variable $\rv{\xi}$ which collects the realizations of the random data of the problem, namely $q\in \R^{n_2}$, $h\in \R^{m_2}$, $W\in \R^{m_2\times n_2}$ and $T\in \R^{m_2\times n_1}$. For given $x$, $Q(x,\rv{\xi})$ is a measurable function from $(\Omega,\Sigma,\mu)$ into the extended real line equipped with the Borel $\sigma$-algebra $\mtc{B}$. The extended real line carries the usual protocol that $Q(x,\xi)=\infty$ captures infeasibility and $Q(x,\xi)=-\infty$ captures unboundedness.

A fundamental assumption in problem \eqref{eq:sp} is that the distribution $\pr{P}$ of $\rv{\xi}$ on $\R^N$ is known and independent of $x$.
This problem has been proven remarkably effective in modeling a vast variety decision problems under uncertainty. 
Many, however, do not fall under this umbrella. 

In this paper we allow the probability distribution to depend on the decision vector $x$. 
Particularly, we consider a finite (but potentially large) collection $(\mtc{X}_d)_{d\in\mtc{D}}$ of subsets of $\mtc{X}$ that defines a partition. 
That is, $\bigcup_{d\in\mtc{D}}\mtc{X}_d=\mtc{X}$ and $\mtc{X}_{d} \cap \mtc{X}_{d'} = \emptyset$ for all $d,d' \in \mtc{D},d \neq d'$.  
We assume that the membership of $x$ to a subset $\mtc{X}_d$ determines a distribution $\pr{P}_d$ on $\R^N$ for the uncertainties $\rv{\xi}$. 
In other words, the non-trivial relationship between $x$ and the resulting distribution is piecewise-constant, i.e., different sub-regions of the feasible space induce different distributions. 
This implies a finite, but potentially very large, number of distributions. 

The choice of working with disjoint subsets of $\mtc{X}$ is driven by our conceptualization of the optimization process. Namely, that the optimizer makes decisions that in turn completely determine the relevant second stage distribution. If we allow the sets $\mtc{X}_d$ to intersect, such a determination can not be made and one is potentially left with the situation of modeling that the optimizer explicitly chooses the relevant second stage distribution. We do not believe such a process reflects the role optimization plays in real-world decision making. That said, if one is able to uniquely determine the relevant distribution for each point in a non-empty intersection, we believe our method applies with minor modifications. However, we do not address this issue in this paper.

We define $\Ind_d:\mtc{X}\to\{0,1\}$ to be the characteristic function of $\mtc{X}_d$. Namely, $\Ind_d(x)= 1$ iff  $x\in\mtc{X}_d$, $0$ otherwise. Given this notation, in this paper we focus on the following stochastic program with decision-dependent uncertainty
\begin{equation}
     \min_{x\in \mtc{X}} ~c^\top x+\sum_{d\in\mtc{D}}\Ind_d(x)\Ex_{\pr{P}_d}[Q(x,\xi)]
     \label{eq:speu}
\end{equation}
Observe that $x\mapsto\sum_{d\in\mtc{D}}\Ind_d(x)=1$. This follows from the fact that the collection $(\mtc{X}_d)_{d\in\mtc{D}}$ is a partition of $\mtc{X}$.  That is, the problem can be viewed as a type of disjunctive program \cite{balas2018disjunctive}.  Observe that in \eqref{eq:speu} infeasibility is tolerated (i.e., $Q(x,\xi)=\infty$ with probability larger than zero) under distributions other than the one enforced by $x$. We assume the problem is bounded from below. 

Problem \eqref{eq:speu} leaves sufficient freedom to model a large variety of configurations. 
For example, the subsets $\mtc{X}_d$ may be convex polyhedra specifying, e.g., admissible ranges of values on individual components of $x$.
If $\mtc{X}$ is a subset of the $n_1$-dimensional integers, subsets $\mtc{X}_d$ may represent subsets of discrete choices, specific combinations, arrangements, permutations, etc.
The resulting distributions $\pr{P}_d$ may in turn be parametrized by a function $\theta:\{0,1\}^{|\mtc{D}|}\to \R^P$ where $P$ is the number of parameters of $\pr{P}_d$, or even assume a different functional form over the different subsets $\mtc{X}_d$. Furthermore, in situations where the distribution is explicitly parameterized by the decision vector, e.g., $\theta\left(x\right)$, this formulation allows approximating the distribution using some reference set of parameters in each subset, e.g., $\theta\left(\bar{x}_d\right)$ for some $\bar{x}_d\in\mtc{X}_d$, $d\in\mtc{D}$.

 An example of an application that can be modeled as Problem \eqref{eq:speu} is a production planning problem in which an organization seeks to determine expected profit-maximizing production volumes in the presence of  uncertainty in both supply and demand. 
Uncertainty in supply is endogeneous and depends on choices regarding the combination of facilities that produce and the respective production volumes. More generally, such a problem is similar to those studied in the literature on lot sizing with random yields \cite{yano1995lot}. We will return to a similar problem in \Cref{sec:comp_setting}.

The primary contribution of this paper is an extension of the classical L-Shaped and Integer L-Shaped methods to efficiently solve problem \eqref{eq:speu}. As the relevant second stage distribution depends on first stage decisions, underlying the method is the idea of \textit{distribution-specific} feasibility and optimality cuts, which we derive and prove. The overall method is general but the form of such cuts depends on the nature of the second stage problem, $Q(x,\xi).$ In particular, we demonstrate how such cuts can be formulated depending on whether $Q(x,\xi)$ forms a linear or a mixed-integer linear program. 
Furthermore, we propose inequalities that are distribution-independent and may speed up convergence. Beyond proving theoretical convergence of the proposed method, we demonstrate computationally on three variants of a production planning problem that it significantly outperforms a state-of-the-art benchmark. 

The remainder of this article is organized as follows. We begin by summarizing the relevant literature on stochastic programs with endogenous uncertainty in \Cref{sec:literature}. We then introduce the method in \Cref{sec:benders}. In \Cref{sec:comp_setting} we introduce three production planning problems we use in the computational study presented in \Cref{sec:comp_study}. Finally, we draw conclusions in \Cref{sec:conclusion_future}.

\section{Literature and contribution}
\label{sec:literature}
According to \cite{GoeG06} there exist at least two ways in which decisions can determine the nature of uncertainties.
The first possibility is that decisions determine the time when the uncertainty is resolved. 
A typical example is provided by \cite{ApaG16}. An oil company has to decide which gas reservoirs to explore and when. Exploration is performed by installing drilling equipment. The size and quality of the reservoirs is uncertain. The uncertainty is resolved only after the sites have been explored. 
Stochastic programs of this kind are studied in a number of articles including \cite{ColM08,TarG08,TarGG09,ColM10,GupG11,MerV11,TarGG13,ApaG16}. 

The second possibility -- which is of interest in this paper -- is that decisions alter the probability space underlying the random variables, hence their distributions. 
The methodological landscape for this type of problems is sparser. 
One of the first approaches dates back to \cite{Pfl90} who consider a Markovian random process whose transition probabilities depend on the decision variables of the optimization problem. The author provides an algorithm that converges with probability one. 
Later, \cite{JonWW98} consider the case where decisions influence both the probability measure and the timing of the observation. 
Their framework includes both two- and multi-stage problems. Nevertheless, the decisions that have an impact on the uncertainty are entirely made at the first decision stage. As in the case described in this article, the authors assume that the set of probability measures which can be enforced by decisions is finite and countable. The authors show that the problem can be recast as choosing the best among the stochastic programs determined by a choice of a distribution and propose an implicit enumeration algorithm.
In a similar framework, \cite{Pan21} considers multi-stage problems and extends the model of \cite{JonWW98} by allowing decisions at all stages to determine the probability measure for the later stages. The author provides a new scenario tree structure and mathematical formulation that avoids explicit statement of non-anticipativity constraints which are typically linked to model size growth, see e.g., \cite{ApaG16,HooM16,MooM18}. 
Furthermore, the author provides an exact solution method for a special case.
Considering two-stage stochastic programs, \cite{Hel16} and \cite{HelBT18} discuss several ways of modeling the interplay between the decision variables and the parameters of the underlying probability distributions. Particularly, the authors formulate two-stage models where prior probabilities are distorted through affine transformations, or combined using convex combinations of several probability distributions.
Furthermore, the authors present models which incorporate the parameters of the probability distribution as first stage decision variables.  Practical applications of problems with endogenous uncertainty can be found in \cite{Ahm00,VisSF04,Fla10,PeeSGV10,LauPK14,HelW05,TonFR12,EscGMU18,hewittpantusoJOC,lejeune2024profit}. We do not discuss these papers in detail for sake of brevity.  

In this article we focus on stochastic programs of the second kind. 
The contribution of this article is twofold. First, we propose a general-purpose model for two-stage stochastic programs with decision-dependent uncertainty. Particularly, we are interested in the case where there exist finitely many distributions that can be enforced by the choice of $x$. 
Second, we provide a general purpose exact algorithm which extends the well-known L-Shaped method. Particularly, optimality and feasibility cuts are derived for problems which may include integer variables at both stages. 

This article extends the available literature in several ways.
In \cite{Hel16,HelBT18} given reference discrete distributions are transformed by first stage decisions through affine transformations.
Hence, their approach can handle all the (infinitely many) discrete distributions that may materialize as an affine transformation of the given reference distributions.
We use a different approach. On the one hand, we restrict ourselves to finitely (though potentially combinatorially) many distributions.
On the other hand, each one of these distributions is arbitrarily dependent on $x$.
In particular, our model allows for arbitrary mappings of (subsets of) first stage decisions to probability distributions.
When the number of probability distributions is naturally finite, our model is exact. This is the case, in particular, when the underlying distribution depends only on specific finite features of the first stage solutions, $x$. When $\mathcal{X}$ is countable and bounded, our model is always exact.

The problem we study is in line with the work of \cite{JonWW98} and \cite{Pan21} in that the set of potential probability distributions enforced by first stage decisions is finite and countable. In particular, in \cite{JonWW98} the authors propose an implicit enumeration algorithm which relies on computing and storing bounds for each choice of a probability distribution. As the authors acknowledge, this solution strategy is viable only when the set of probability distributions is small. 
As will be more evident in \Cref{sec:model}, this may in general not be the case in large scale problems. To overcome these methodological limitations, we extend the L-Shaped method by deriving novel distribution-specific optimality and feasibility cuts. The method is generally applicable to stochastic programs, possibly with integer variables at both stages.

\section{The L-shaped method for problems with endogenous uncertainty}
\label{sec:benders}
In this section, we adapt the L-Shaped method to provide an exact algorithm for problem \eqref{eq:speu}. As part of that adaptation we derive new distribution-specific optimality and feasibility cuts.  We note that the operations of the proposed method are based on the assumption that problem \eqref{eq:speu} is feasible and bounded from below. We treat, separately, two important cases. The first case, discussed in \Cref{sec:benders:cont}, is when $\mtc{Y}\subseteq \R^{n_2}$. In the second, discussed in \Cref{sec:benders:int},  $\mathcal{Y}\subseteq\mathbb{Z}^{n_2^1}\times \R^{n_2^2}, n_2^1+n_2^2=n_2$.  For each case we show there are familes of functions, $\mtc{F} = \{\mtc{F}_{1},\ldots,\mtc{F}_{\vert \mtc{D} \vert} \}, \mtc{O} = \{\mtc{O}_{1},\ldots,\mtc{O}_{\vert \mtc{D} \vert} \},$ that induce constraints such that the following is an equivalent reformulation of \eqref{eq:speu}.
\begin{subequations}\label{eq:speu_reform}
\begin{align}
     \min_{x\in \mtc{X}} & ~c^\top x+\mu \\
     x \in \mtc{X}_{d}  \implies &
     \begin{cases}
     	f_{d}(x) \leq 0 \;\;\; \forall f_{d}(\cdot) \in \mtc{F}_{d} \\
	o_{d}(x) \leq \mu \;\; \forall o_{d}(\cdot) \in \mtc{O}_{d}
     \end{cases} \forall d\in\mtc{D}
\end{align}
\end{subequations}
Namely, given an optimal solution $(x^*,\mu^*)$ to \eqref{eq:speu_reform}, $x^*$ is optimal for \eqref{eq:speu} and $\mu^* = \sum_{d\in\mtc{D}}\Ind_d(x^*)\Ex_{\pr{P}_d}[Q(x^*,\xi)].$ We refer to the set $\mtc{F}_{d}$ as a set of distribution-specific (to $d$) feasibility cuts and show that $\mtc{F}_{d}$ is finite. We show the same for the set of distribution-specific optimality cuts, $\mtc{O}_{d}.$ 

Similar to the classical L-Shaped method, we consider subsets $\underline{\mtc{F}_{d}} \subseteq \mtc{F}_{d}, \underline{\mtc{O}_{d}} \subseteq \mtc{O}_{d}, d \in \mtc{D}$ of each set of functions and formulate the following \textit{Relaxed Master Problem} (RMP).
\begin{subequations}\label{eq:rmp}
\begin{align}
     \min_{x\in \mtc{X}} & ~c^\top x+\mu \\
     \label{eq:rmp:imp}x \in \mtc{X}_{d}  \implies &
     \begin{cases}
     	f_{d}(x) \leq 0 \;\;\; \forall f_{d}(\cdot) \in \underline{\mtc{F}_{d}}, \\
	o_{d}(x) \leq \mu \;\; \forall o_{d}(\cdot) \in \underline{\mtc{O}_{d}}.
     \end{cases}
\end{align}
\end{subequations}

A high-level overview of the steps the method executes at an iteration is presented in Algorithm \ref{alg:ls}. 
At iteration $i$, the method solves problem \eqref{eq:rmp} formulated with sets $\underline{\mtc{F}_{d}^{i}},  \underline{\mtc{O}_{d}^{i}}$, to obtain the solution $(x^i,\mu^i)$.
We note that, at the initial iterations, without a valid lower bound for $\mu$, it may be necessary to exclude $\mu$ from optimization and set its value to $\infty$ to prevent unboundedness. 
The method then identifies the distribution $d^i$ enforced by $x^i$. 
Next, the method checks whether $x^i$ yields second stage feasible subproblems under distribution $d^i$.
If it does not, the method identifies a feasibility cut from $\mtc{F}_{d^i} \setminus \underline{\mtc{F}_{d^i}^i}$ and adds it to $\underline{\mtc{F}_{d^i}^{i+1}}$.

If the feasibility check is passed, the method moves on to assess the accuracy of $\mu^i,$ the estimate of expected second stage costs.
To do so, it computes the expected second stage cost $\Ex_{\pr{P}_{d^i}}[Q(x^i,\xi)]$ under the distribution $d^i$ enforced by $x^i$. 
If the approximation $\mu^i$ is lower than this expected value it identifies an optimality cut from  $\mtc{O}_{d^i} \setminus \underline{\mtc{O}_{d^i}^i}$ and adds it to $\underline{\mtc{O}_{d^i}^{i+1}}$.
If the optimality check is passed, the solution $x^i$ is provably optimal and the method terminates. 

\begin{algorithm}[htp]
\caption{L-Shaped Method for \eqref{eq:speu}}\label{alg:ls}
\begin{algorithmic}
\State $i\gets 0$, \texttt{solved}$\gets$\texttt{false}
\While{not \texttt{solved}}
\State Solve \eqref{eq:rmp}. Let $(x^i,\mu^i)$ be its optimal solution
\State Identify $d^i\in\mtc{D}$ such that $x^i\in\mtc{X}_{d^i}$
\If{$x^i$ is second stage infeasible}
    \State Add a feasibility cut from $\mtc{F}_{d^i} \setminus \underline{\mtc{F}_{d^i}^{i}}$ to $\underline{\mtc{F}_{d^i}^{i+1}}$
\ElsIf{$\mu^i<\Ex_{\pr{P}_{d^i}}[Q(x^i,\xi)]$ }
    \State Add an optimality cut from $\mtc{O}_{d^i} \setminus \underline{\mtc{O}_{d^i}^i}$ to $\underline{\mtc{O}_{d^i}^{i+1}}$
\Else   
    \State \texttt{solved}$\gets$\texttt{true}
\EndIf
\State $i\gets i+1$
\EndWhile
\end{algorithmic}
\end{algorithm}
We note that the method described in Algorithm \ref{alg:ls} is amenable to well-known enhancements of the classical L-Shaped method.
For example, as \eqref{eq:rmp} is defined over subsets of $\mtc{F}$ and $\mtc{O}$ it is a relaxation of \eqref{eq:speu_reform} (and hence \eqref{eq:speu}). Hence, its objective value provides a lower bound on the optimal objective value to \eqref{eq:speu}. 
Also, we note that when $x^i$ passes the feasibility check, $c^\top x^i + \Ex_{\pr{P}_{d^i}}[Q(x^i,\xi)]$ is an upper bound on the optimal objective value to \eqref{eq:speu}. 
The method may then terminate when the gap between the two bounds falls below a pre-defined threshold.  
Furthermore, the implications \eqref{eq:rmp:imp} imply that \eqref{eq:rmp} is a MILP. Therefore, Algorithm \ref{alg:ls} can be invoked at nodes in a branch-and-bound procedure that solves \eqref{eq:rmp}, see \cite{LapL93}.

The following sections present appropriate familes of functions $\mtc{F}, \mtc{O}$ for different classes of second stage problems.
For the sake of generality and ease of exposition, the constraints based on those functions are expressed with indicator functions, $\Ind_d(x)$. 
Note, however, that a direct implementation of such indicator functions may require a large number of binary variables (at least $|\mtc{D}|$) in \eqref{eq:rmp}. At the end of the section we will show how, under rather general conditions, the implications $x\in\mtc{X}_d\implies f_{d}(x)\leq 0$ and $x\in\mtc{X}_d\implies o_{d}(x) \leq \mu$ can be modeled using fewer binary variables.

\subsection{Continuous second stage}\label{sec:benders:cont}
In this section, we consider the case in which the second stage problem can take the form of a linear program. 
Namely

\begin{equation}\label{eq:sp2s:continuous}
Q(x,\xi)=\min_{y\geq 0}\{q^\top y|Wy=h-Tx\}.
\end{equation}

The validity of the proposed optimality and feasibility cuts require the following assumption.

\begin{itemize}
    \item[\textbf{A1}] For each $d\in\mtc{D}$, the uncertainty $\rv{\xi}$ has discrete distribution $\pr{P}_d=\sum_{s\in\mtc{S}_d}\pi_{sd}\Delta_s$, supported by $\Xi_d$ $=$ $\{$$\xi_{1d}$,$\ldots$,$\xi_{Sd}\}$ where $\Delta_s$ is the Dirac measure that assigns total mass to realization $\xi_{sd}$ for $s$ in the set of scenarios $\mtc{S}_d$  and $\pi_{sd}$ is the probability of realization $\xi_{sd}$ under distribution $d\in\mtc{D}$. 
\end{itemize}

When $\rv{\xi}$ has a continuous distribution, assumption A1 can be satisfied by means of a discretization technique. In what follows, we begin by introducing distribution-specific feasibility cuts for problems without complete recourse. We then continue by introducing distribution-specific optimality cuts.

\subsubsection{Distribution-specific feasibility cuts}
\label{subsubsec:cont_sp:feas_cut}
When complete recourse is not present, the method must be augmented with a technique for generating feasibility cuts to ensure convergence. We next define, for all $d\in\mtc{D}$ and $s\in\mtc{S}_d$ the set 
$$\mtc{K}_{sd}=\big\{x\in\R^{n_1}|\exists y\geq 0: W_{sd}y=h_{sd}-T_{sd}x \big\}.$$
In words, $\mtc{K}_{sd}$ contains values of $x$ such that the corresponding subproblem for scenario $s$ of distribution $d$ is feasible. Relatedly, we define for all $d\in\mtc{D}$, $\mtc{K}_{d}=\bigcap_{s\in\mtc{S}_d}\mtc{K}_{sd}$.
This allows us to define relatively complete recourse more precisely as $\mtc{X}\subseteq \bigcap_{d\in\mtc{D}}\mtc{K}_d$ and complete recourse as $\mtc{K}_d=\R^{n_1}$ for all $d\in\mtc{D}$.
Observe, that we rule out the possibility that $\mtc{K}_d=\emptyset$ given that we assume problem \eqref{eq:speu} is feasible.

To define feasibility cuts, we make, in addition to A1, the following assumption.
\begin{itemize}
    \item[\textbf{A2}] For all $d\in\mtc{D}$ and $s\in\mtc{S}_d$, there exists a constant $U_{ds}<\infty$ such that
    $$U_{ds}\geq \max_{x\in\mtc{X}}\min_{y\in\R^{n_2}_+,w^+,w^-\in\R^{m_2}_+}\{\mathbf{1}^\top w^++\mathbf{1}w^-:W_{sd^v}y+w^+-w^-=h_{sd^v}-T_{sd^v}x\}$$
    where $\mathbf{1}^\top=(1,\ldots,1)\in\R^{m_2}$.
\end{itemize}
The quantities $U_{ds}$ enable the definition of feasibility cuts that are distribution-specific. Namely, they are binding only for first stage solutions that induce the same distribution for which they were generated and redundant otherwise.

\begin{theorem}\label{prop:fc:continuous}
Let $(x^v,\mu^v)$ be a solution to \eqref{eq:rmp} that induces distribution $d^v$ such that $x^v\notin\mtc{K}_{sd^v}$ for some $s\in\mtc{S}_{d^v}$.
Then, solution $x^v$ violates inequality 
\begin{equation}\label{eq:continuous:fc}
    (\rho^v)^\top(h_{sd^v}-T_{sd^v}x)\leq U_{d^vs}(1-\Ind_{d^v}(x^v))
\end{equation}
where $\rho^v$ is an optimal solution to the the dual of 
\begin{equation}\label{eq:continuous:fsp}
\Psi(x^v,\xi_{sd^v}):=\min_{y\in\R^{n_2}_+,w^+,w^-\in\R^{m_2}_+}\{\mathbf{1}^\top w^++\mathbf{1}w^-:W_{sd^v}y+w^+-w^-=h_{sd^v}-T_{sd^v}x^v\}
\end{equation}
\end{theorem}
\begin{proof}
We note that $\Psi(x^v,\xi_{sd^v})$  is a feasible problem and that $x^v\notin\mtc{K}_{sd^v}$ implies that $\Psi(x^v,\xi_{sd^v})  >0.$
Letting $\rho^v$ be an optimal solution to the dual of $\Psi(x^v,\xi_{sd^v}),$ we have that $\Psi(x^v,\xi_{sd^v})= (\rho^v)^\top(h_{sd^v}-T_{sd^v}x^v)>0.$
In turn, given that $\Ind_{d^v}(x^v)=1$ we have
$$\Psi(x^v,\xi_{sd^v})= (\rho^v)^\top(h_{sd^v}-T_{sd^v}x^v)> 0 = U_{d^vs}(1-\Ind_{d^v}(x^v))$$
as required.
\end{proof}

The following theorem shows that the proposed feasibility cuts are safe, in the sense that a cut generated for a solution $x^v$ to \eqref{eq:rmp} is satisfied by solutions $(x^l,\mu^l$) to \eqref{eq:rmp}, for which $x^l \neq x^v$, and the subproblems are feasible in the distribution enforced by $x^l$. 
\begin{theorem}\label{prop:fc:continuous:safe}
Let $(x^l,\mu^l)$ be a solution to \eqref{eq:rmp} such that $x^l\in\mtc{K}_{d^l}$.
Then, solution $x^l$ satisfies inequality
\begin{equation*}
    (\rho^v)^\top(h_{sd^v}-T_{sd^v}x)\leq U_{d^vs}(1-\Ind_{d^v}(x))
\end{equation*}
where $\rho^v$ is an optimal solution to the dual of $\Psi(x^v,\xi_{sd^v})$ for some solution $x^v$ to \eqref{eq:rmp} and $s\in\mtc{S}_{d^v}$.
\end{theorem}

\begin{proof}
We consider separately the cases in which $d^l\neq d^v$ and in which $d^l=d^v$.

When $d^l\neq d^v$, we have $\Ind_{d^v}(x^l)=0$ and inequality \eqref{eq:continuous:fc} reduces to 
$$(\rho^v)^\top(h_{sd^v}-T_{sd^v}x)\leq U_{d^vs}$$
Consider the set $\mtc{R}_{sd}:=\left\{\rho\in\R^{m_2}\left\vert \rho^\top W_{sd}\leq 0,-\mathbf{1}\leq\rho\leq \mathbf{1}\right.\right\}$
so that $\rho^v\in \argmax_{\rho\in\mtc{R}_{sd^v}} \{\rho^\top(h_{sd^v}-T_{sd^v}x^v)\}$. 
Then we have
\begin{align*}
  (\rho^v)^\top(h_{sd^v}-T_{sd^v}x^l)\leq& \max_{\rho\in\mtc{R}_{sd^v}}\rho^\top(h_{sd^v}-T_{sd^v}x^l)\\
  \leq &\max_{x\in\mtc{X}}\max_{\rho\in\mtc{R}_{sd^v}}\rho^\top(h_{sd^v}-T_{sd^v}x)\\
  =&\max_{x\in\mtc{X}}\min_{y\in\R^{n_2}_+,w^+,w^-\in\R^{m_2}_+}\{\mathbf{1}^\top w^++\mathbf{1}w^-:W_{sd^v}y+w^+-w^-=h_{sd^v}-T_{sd^v}x\}\\
  \leq & U_{d^vs}
\end{align*}
Hence, the inequality is satisfied.

When $d^l=d^v$, as $x^l\in\mtc{K}_{d^l}$ we have $x^l\in\mtc{K}_{sd^l}=\mtc{K}_{sd^v}$.
Consequently, 
$$\Psi(x^l,\xi_{sd^v})=\min_{y\in\R^{n_2}_+,w^+,w^-\in\R^{m_2}_+}\{\mathbf{1}^\top w^++\mathbf{1}w^-:W_{sd^v}y+w^+-w^-=h_{sd^v}-T_{sd^v}x^l\}=0$$
By strong duality of $\Psi(x^l,\xi_{sd^v})$ we can write
$$0=\Psi(x^l,\xi_{sd^v})=(\rho^l)^\top(h_{sd^v}-T_{sd^v}x^l)$$
where $\rho^l$ is an optimal solution to the dual of $\Psi(x^l,\xi_{sd^v})$.
As $\rho^l$ is optimal for the dual while $\rho^v$ is feasible we have, as required, that
$$U_{d^vs}(1-\Ind_{d^v}(x^l))=0=\Psi(x^l,\xi_{sd^v})=(\rho^l)^\top(h_{sd^v}-T_{sd^v}x^l)\geq (\rho^v)^\top(h_{sd^v}-T_{sd^v}x^l)$$

\end{proof}

The following proposition shows that the method will deliver second stage feasible solutions within a finite number of iterations. 
\begin{lemma}\label{prop:fc:continuous:finite}
 There exist only finitely many cuts \eqref{eq:continuous:fc}.
\end{lemma}
\begin{proof}
This follows immediately from the finite number of distributions $\mtc{D}$, finite number of realizations $\mtc{S}_d$ for all $d\in\mtc{D},$ and finite number of extreme points $\rho$ in the feasible region of the dual to $\Psi(x,\xi_{sd})$ for all $s\in\mtc{S}_d$ and $d\in\mtc{D}$.
\end{proof}

\subsubsection{Distribution-specific optimality cut}
\label{subsubsec:cont_sp:opt_cut}
The following proposition introduces a distribution-specific, duality-based, optimality cut. The validity of this cut relies on the existence of a constant $U$ that satisfies the following inequalities. 

\begin{itemize}
    \item[\textbf{A3}] There exists a constant $U$ such that  $\infty>U\geq \max_{d\in\mtc{D}}U_d-\min_{d\in\mtc{D}}L_d$ where
    $$U_d\geq \max_{x\in\mtc{X}_{d}}\sum_{s\in\mtc{S}_{d}}\pi_{sd}Q(x,\xi_{sd})$$
    $$L_d\leq \min_{x\in\mtc{X}_{d}}\sum_{s\in\mtc{S}_{d}}\pi_{sd}Q(x,\xi_{sd})$$

\end{itemize}

We note that such a $U$ exists when the second stage problem is bounded above and below, when feasible.

\begin{theorem}\label{prop:oc:continuous}
Let $(x^v,\mu^v)$ be a solution to \eqref{eq:rmp} and
$d^v\in\mtc{D}$ such that $x^v\in\mtc{X}_{d^v}$. Assume $x^v\in\mtc{K}_{d^v}$ and that $\mu^v < Q_{d^v}(x^v)$ with  
$$Q_{d^v}(x^v)=\sum_{s\in\mtc{S}_{d^v}}\pi_{sd}Q(x^v,\xi_{sd^v})$$ 
Then, solution $(x^v,\mu^v)$ violates inequality
\begin{equation}\label{eq:continuous:oc}
    \mu\geq \sum_{s\in\mtc{S}_{d^v}}\pi_{sd^v}(\rho^v_s)^\top (h_{sd^v}-T_{sd^v}x) - U(1-\Ind_{d^v}(x))
\end{equation}
where $\rho^v_s$ is an optimal solution to the dual of $Q(x^v,\xi_{sd^v})$.
\end{theorem}
\begin{proof}
Since $\mu^v < Q_{d^v}(x^v)$ we can write
\begin{align*}
    \mu^v &< Q_{d^v}(x^v)=\sum_{s\in\mtc{S}_{d^v}}\pi_{sd^v}Q(x^v,\xi_{sd^v})\\
    &=\sum_{s\in\mtc{S}_{d^v}}\pi_{sd^v}\bigg((\rho^v_s)^\top(h_{sd^v}-T_{sd^v}x) \bigg)\\
    &=\sum_{s\in\mtc{S}_{d^v}}\pi_{sd^v}\bigg((\rho^v_s)^\top(h_{sd^v}-T_{sd^v}x)\bigg)- U(1-\Ind_{d^v}(x^v))
\end{align*}
where the second equality holds by strong duality of $Q(x^v,\xi_{sd^v})$ (recall that $x^v\in\mtc{K}_{d^v}$) and the third equality holds since $\Ind_{d^v}(x^v)=1$.
\end{proof}

\Cref{prop:oc:continuous} introduces an optimality cut that renders infeasible in \eqref{eq:rmp} a pair $(x^v,\mu^v)$ whenever $\mu^v$ does not hold the correct expected cost conditional on $x^v$.
It is well-know that, given $d$, $Q_d(x)$ is convex and piece-wise linear, see e.g., \cite{WalW67}. Furthermore, $\sum_{s\in\mtc{S}_d}\pi_{sd}\rho^v_s(h_{sd}-T_{sd}x)$ defines a supporting hyperplane to the epigraph of $Q_d(x)$ at $x^v$. The last term in the optimality cut \eqref{eq:continuous:oc} ensures that the cut is effective only when $x$ enforces distribution $d^v$. 

Having established that the inequality renders the current solution $(x^v,\mu^v)$ infeasible, we next prove that the inequality is satisfied by solutions $(x^l, \mu^l)$, $x^l \neq x^v,$ such that $\mu^l$ does not under-estimate expected second stage costs.

\begin{theorem}\label{prop:oc:continuous:safe}
Let $(x^l,\mu^l)$ be a solution to \eqref{eq:rmp} such that $x^l\in\mtc{K}_{d^l}$ and for which $\mu^l = Q_{d^l}(x^l)$. 
Here $d^l\in\mtc{D}$ is such that $x^l\in\mtc{X}_{d^l}$. 
Then, solution $(x^l,\mu^l)$ satisfies inequality
\begin{equation}
    \mu\geq \sum_{s\in\mtc{S}_{d^v}}\pi_{sd^v}(\rho^v_s)^\top(h_{sd^v}-T_{sd^v}x) - U(1-\Ind_{d^v}(x))
\end{equation}
where $\rho^v_s$ is an optimal solution to the dual of $Q(x^v,\xi_{sd^v})$ and $d^v$ the distribution enforced by some solution $x^v$ to \eqref{eq:rmp}.
\end{theorem}
\begin{proof}
There are two cases to consider. The first, $d^l \neq d^v,$ reflects when the distribution enforced by solution $x^l$ is different from the distribution enforced by the solution $x^v$ for which the inequality was generated. The second, $d^l = d^v$, is when the two distributions are the same. 

When $d^l \neq d^v$ we have 
\begin{align*}
    \mu^l &\geq \sum_{s\in\mtc{S}_{d^v}}\pi_{sd^v}\bigg((\rho^v_s)^\top(h_{sd^v}-T_{sd^v}x^l)\bigg)- U\\
    &=\sum_{s\in\mtc{S}_{d^v}}\pi_{sd^v}\bigg((\rho^v_s)^\top(h_{sd^v}-T_{sd^v}x^l)\bigg)- U(1-\Ind_{d^v}(x^l))
\end{align*}
where the inequality holds due to the definition of $U$, since 
$$\sum_{s\in\mtc{S}_{d^v}}\pi_{sd^v}\bigg((\rho^v_s)^\top(h_{sd^v}-T_{sd^v}x^l)\bigg)\leq \max_{x\in\mtc{X}}\sum_{s\in\mtc{S}_{d^v}}\pi_{sd^v}Q(x,\xi_{sd^v})$$
and 
$$\mu^l = Q_{d^l}(x^l)\geq \min_{x\in\mtc{X}}\sum_{s\in\mtc{S}_{d^l}}\pi_{sd^l}Q(x,\xi_{sd^l})$$
and the equality holds since the case implies that $\Ind_{d^v}(x^l)=0$.

When $d^l=d^v$ we have, given that $\rho^l_s$ is an optimal solution to the dual of $Q(x^l,\xi_{sd^l})$,
\begin{align*}
    \mu^l &=\sum_{s\in\mtc{S}_{d^l}}\pi_{sd^l}Q(x^l,\xi_{sd^l})=\sum_{s\in\mtc{S}_{d^v}}\pi_{sd^v}Q(x^l,\xi_{sd^v})\\
    &=\sum_{s\in\mtc{S}_{d^v}}\pi_{sd^v}\bigg((\rho^l_s)^\top(h_{sd^v}-T_{sd^v}x^l)\bigg)\\
    &\geq \sum_{s\in\mtc{S}_{d^v}}\pi_{sd^v}\bigg((\rho^v_s)^\top(h_{sd^v}-T_{sd^v}x^l)\bigg)\\
    &= \sum_{s\in\mtc{S}_{d^v}}\pi_{sd^v}\bigg((\rho^v_s)^\top(h_{sd^v}-T_{sd^v}x^l)\bigg)-U(1-\Ind_{d^v}(x^l))
\end{align*}
The second equality holds because the distributions $d^l,d^v$ are the same. The third equality holds due to strong duality of $Q(x^l,\xi_{sd^l})$. 
The inequality follows from dual optimality of $\rho^l_s$ and dual feasibility of $\rho^v$. The last equality holds as $\Ind_{d^v}(x^l)=\Ind_{d^l}(x^l)=1$.
\end{proof}

\begin{lemma}\label{prop:oc:continuous:finite}
 There exist only finitely many optimality cuts \eqref{eq:continuous:oc}.
\end{lemma}
\begin{proof}
This follows immediately from the finite number of distributions $\mtc{D}$ and from the finite number of extreme points $\rho$ in the feasible region of the dual to $Q(x,\xi)$ for each $\xi$.
\end{proof}

\subsubsection{Reformulation}
In this section we present a reformulation of \eqref{eq:speu} in terms of the feasibility and optimality cuts just proposed. We also prove the validity of this reformulation. Namely, that an optimal solution to this reformulation can be used to derive an optimal solution to \eqref{eq:speu} with the same objective function value.

The reformulation is as follows 

\begin{align}\label{eq:cont_reform}
     \min_{x\in \mtc{X}}&  ~c^\top x+\mu  \\
     &U_{sd}(1-\Ind_{d}(x)) \geq  (\sigma^v_s)^\top(h_{sd}-T_{sd}x)   \;\;\; \forall d \in \mtc{D}, s \in \mathcal{S}_{d}, \sigma^{v}_s \in F_{d}^{s}, \nonumber \\
	 &\mu\geq  \sum_{s\in\mtc{S}_{d}}\pi_{sd}(\rho_s^v)^\top (h_{sd}-T_{sd}x) - U(1-\Ind_{d}(x)) \;\;\; \forall d \in \mtc{D}, \rho^{v} \in O_{d} \nonumber.
\end{align}
where $F_{d}^{s}$ is the set of extreme points of the polyhedron $\{\sigma: \sigma W_{sd} \leq 0, \mathbf{-1} \leq \sigma \leq \mathbf{1} \}$,
$\rho^v=(\rho_s^v)_{s\in\mtc{S}_d}$, $O_{d}^{s}$ is the set of extreme points of the polyhedron $\{\rho: \rho W_{sd} \leq q_{sd}\}$, $O_d=O_d^1\times O_d^2 \times \cdots \times O_d^{\vert\mtc{S}_d\vert}$, and $U_{sd}$ and $U$ are defined as above. 
We note that, given \Cref{prop:fc:continuous:finite} and \Cref{prop:oc:continuous:finite}, there are a finite number of feasibility and optimality cuts used to define this formulation. 

\begin{proposition}\label{prop:equivalent_reform}
An optimal solution $(x^*,\mu^*), x^* \in \mathcal{X}_{d^*}$ of \eqref{eq:cont_reform} induces an optimal solution $x^*$ to \eqref{eq:speu} of value $c^{T}x^* + \mu^*$ and $\mu^*=\Ex_{\pr{P}_{d^*}}[Q(x^*,\xi)]$. 
\end{proposition}

\begin{proof}
The proof can be obtained by contradiction. First, assume that the solution $x^*$ to \eqref{eq:cont_reform} is such that $x^*\notin\mtc{K}_{d^*}$. In other words, for the relevant distribution $d^*$, $x^*$ induces an infeasible subproblem for some scenario $s \in \mathcal{S}_{d^*}.$ However, by \Cref{prop:fc:continuous} this implies a feasibility cut of the form \eqref{eq:continuous:fc} can be generated and added to \eqref{eq:cont_reform}, contradicting the premise that \eqref{eq:cont_reform} is formulated with all such (finitely many) cuts \eqref{eq:continuous:fc}. Next, suppose $\mu^* < \sum_{d\in\mtc{D}}\Ind_d(x^*)\Ex_{\pr{P}_d}[Q(x^*,\xi)] = \Ex_{\pr{P}_{d^*}}[Q(x^*,\xi)]= Q_{d^*}(x^*).$ However, by \Cref{prop:oc:continuous} this implies an optimality cut of the form \eqref{eq:continuous:oc} can be generated and added to \eqref{eq:cont_reform}, contradicting the premise that \eqref{eq:cont_reform} is formulated with all such (finitely many) cuts.
\end{proof}

\begin{theorem}\label{prop:convergence}
Assume \eqref{eq:speu} is feasible and bounded. Then the L-Shaped algorithm with distribution-specific cuts converges to an optimal solution in a finite number of iterations.  
\end{theorem}
  \begin{proof}
    The observation that there are a finite number of cuts of the form  \eqref{eq:continuous:fc} and \eqref{eq:continuous:oc}, coupled with \Cref{prop:equivalent_reform}, provides a sufficient condition for the convergence of an L-Shaped method that iteratively adds, and does not remove, cuts of those forms to \eqref{eq:rmp}.
  \end{proof}

\subsection{Integer second stage}\label{sec:benders:int}
In this section, we consider the case where the expected second stage cost may take an arbitrary form.
That is, given $x$, the quantity $Q_d(x)=\Ex_{\pr{P}_d}[Q(x,\xi)]$ can be computed.
A case of practical interest is when the second stage problem takes the form of a Mixed Integer Linear Program (MILP), that is 
\begin{equation}\label{eq:sp2s:integer}
Q(x,\xi)=\min_{y}\{q^\top~ y|Wy=h-Tx, y\in\mtc{Y}\}
\end{equation}
and $\mtc{Y}$ imposes integrality restrictions on at least one second stage variable. As such, duality theory is not available to generate the cuts presented above for the case of a continuous second stage. 
The following assumption is required for both the optimality and feasibility cuts proposed in this section. 
\begin{itemize}
    \item[\textbf{A4}] $\mtc{X}=\{0,1\}^{n_1}$.
\end{itemize}
In other words, the first stage decision variables that have an impact on second stage costs must take on binary values. Relatedly, both inequalities rely on identifying the set $\mtc{I}^v=\{i\in\{1,\ldots,n_1\}|x^v_i=1\}$ for solutions $x^v$ to \eqref{eq:rmp}.

\subsubsection{Distribution-specific feasibility cut}
\label{subsubsec:mip_sp:feas_cut}

The inequality we present in this section can render any solution infeasible in \eqref{eq:rmp}, including those that induce an infeasible second stage subproblem of the relevant distribution. 
As for the case of continuous second stage problems, we define for all $d\in\mtc{D}$ and $s\in\mtc{S}_d$ the set 
$$\mtc{K}_{sd}=\big\{x\in\R^{n_1}|\exists y\in\mtc{Y}: W_sy=h_s-T_sx \big\}.$$
Similarly, for all $d\in\mtc{D}$, we let $\mtc{K}_{d}=\bigcap_{s\in\mtc{S}_d}\mtc{K}_{sd}$. We then generate cuts of the form \eqref{eq:fc:binary} for solutions $x^v\notin \mtc{K}_{d^v}$.

\begin{theorem}\label{prop:fc:binary}
Let $x^v$ be a solution to \eqref{eq:rmp}.
Then solution $x^v$ violates the following inequality.
\begin{equation}
\label{eq:fc:binary}
    \sum_{i\in\mtc{I}^v}x_i-\sum_{i\notin\mtc{I}^v}x_i\leq |\mtc{I}^v|-1
\end{equation}
Furthermore, the inequality is satisfied by any $x^l\neq x^v$.
\end{theorem}
\begin{proof}
As $x_i^v = 1 \;\forall i \in \mtc{I}^v$ and $x_i^v = 0 \;\forall i \not\in \mtc{I}^v$, we have that
$$\sum_{i\in\mtc{I}^v}x_i^v-\sum_{i\notin\mtc{I}^v}x_i^v= |\mtc{I}^v|.$$ Thus, $x^v$ violates the proposed inequality.

Next, consider a solution $x^l \neq x^v$. As either $\exists i \in \mtc{I}^v$ such that $x_i^l = 0$ or $\exists i \not\in \mtc{I}^v$ such that $x_i^l = 1$, we always have that
$$\sum_{i\in\mtc{I}^v}x_i^l-\sum_{i\notin\mtc{I}^v}x_i^l<|\mtc{I}^v|.$$ 
\end{proof}

Observe, however, that unlike cuts \eqref{eq:continuous:fc}, cuts of type \eqref{eq:fc:binary} only render infeasible a single solution. 

\subsubsection{Distribution-specific optimality cut}
\label{subsubsec:mip_sp:opt_cut}
We next propose an optimality cut that renders infeasible a solution $(x^*,\mu^*)$ such that $\mu^*$ does not correctly estimate expected second stage costs.
To derive optimality cuts for this case we extend the optimality cuts for integer stochastic programs \cite{LapL93}, which rely on assumption A4 and A5 introduced next. 
 
 \begin{itemize}
     \item[\textbf{A5}] There exists a constant $U$ such that $\infty>U\geq \max_{d\in\mtc{D}}L_d-\min_{d\in\mtc{D}}L_d$ where $L_d\leq \min_{x\in\mtc{X}}\sum_{s\in\mtc{S}_{d}}\pi_{sd}Q(x,\xi_s)$

 \end{itemize}

The optimality cut is as follows. 
\begin{theorem}\label{prop:oc:int}
Let $(x^v,\mu^v)$ be a solution to \eqref{eq:rmp} for which $\mu^v<Q_{d^v}(x^v)$. 
Then, $(x^v,\mu^v)$ violates the following inequality.
    \begin{equation}\label{eq:ls:int:oc}
        \mu\geq (Q_{d^v}(x^v)-L_{d^v})\left(\sum_{i\in\mtc{I}^v}x_i-\sum_{i\notin\mtc{I}^v}x_i\right)-\left(Q_{d^v}(x^v)-L_{d^v}\right)\left(|\mtc{I}^v|-1\right)+L_{d^v}-U\left(1-\Ind_{d^v}(x)\right)
    \end{equation}
\end{theorem}
\begin{proof}
For $x=x^v$ we have $\Ind_{d^v}(x^v)=1$ and the inequality reduces to 
$$\mu\geq \left(Q_{d^v}(x^v)-L_{d^v}\right)\left(\sum_{i\in\mtc{I}^v}x_i-\sum_{i\notin\mtc{I}^v}x_i\right)-\left(Q_{d^v}(x^v)-L_{d^v}\right)\left(|\mtc{I}^v|-1\right)+L_{d^v}$$
Following Proposition 2 in \cite{LapL93} for $x=x^v$, the right-hand-side reduces to $Q_{d^v}(x^v)$. Hence, the inequality is violated given the premise that
 $\mu^v<Q_{d^v}(x^v)$.
\end{proof}
The following proposition shows that the inequality is not violated by solutions $(x^l,\mu^l)$ such that $\mu^l$ does not under-estimate expected second stage costs.
\begin{theorem}\label{prop:oc:int:safe}
Let $(x^l,\mu^l)$ be a solution to \eqref{eq:rmp} for which $\mu^l=Q_{d^l}(x^l)$. 
Then, $(x^l,\mu^l)$ satisfies optimality cut \eqref{eq:ls:int:oc} generated for some $x^v, x^v \neq x^l$.
\end{theorem}
\begin{proof}
We again consider the cases $d^l\neq d^v$ and $d^l=d^v$ separately.

When $d^l\neq d^v$ we have $\Ind_{d^v}(x^l)=0$ and the inequality reduces to
$$\mu^l\geq \left(Q_{d^v}(x^v)-L_{d^v}\right)\left(\sum_{i\in\mtc{I}^v}x^l_i-\sum_{i\notin\mtc{I}^v}x^l_i\right)-\left(Q_{d^v}(x^v)-L_{d^v}\right)\left(|\mtc{I}^v|-1\right)+L_{d^v}-U$$
Letting $(\mtc{I}^v)^C$ represent the complement of $\mtc{I}^v$ to $\mtc{I}$ we have that 
$$-|(\mtc{I}^v)^C|\leq \sum_{i\in\mtc{I}^v}x^l_i-\sum_{i\notin\mtc{I}^v}x^l_i\leq |\mtc{I}^v|-1$$
Given this, the inequality reduces to $\mu^l\geq \alpha+L_{d^v}-U$ wherein $\alpha$ is such that $$-(Q_{d^v}(x^v)-L_{d^v})(|\mtc{I}|-1)\leq \alpha\leq 0 $$
By subtraction of $U$ the inequality is always satisfied, even for $\alpha=0$.

When $d^l= d^v$, given that $x^v\neq x^l$ we have $\Ind_{d^v}(x^l)=1$ and the inequality reduces to
$$\mu\geq \left(Q_{d^v}(x^v)-L_{d^v}\right)\left(\sum_{i\in\mtc{I}^v}x^l_i-\sum_{i\notin\mtc{I}^v}x^l_i\right)-\left(Q_{d^v}(x^v)-L_{d^v}\right)\left(|\mtc{I}^v|-1\right)+L_{d^v}$$
Then, the inequality becomes $\mu\geq \alpha+L_{d^v}$ which, given the definition of $L_{d^v},$ and that $\alpha\leq 0$, is always satisfied for distribution $d^v$. 
\end{proof}

\subsubsection{Reformulation}
In this section we present a reformulation of \eqref{eq:speu} in terms of the feasibility and optimality cuts just proposed for the case of second stage problems that take the form of MILPs.
We also prove that an optimal solution to this reformulation can be used to derive an optimal solution to \eqref{eq:speu} with the same objective function value. The reformulation is as follows. 

\begin{align}\label{eq:int_reform}
     \min_{x\in \mtc{X}} &  ~c^\top x+ \mu  \\
  &\sum_{i\in\mtc{I}^v}x_i-\sum_{i\notin\mtc{I}^v}x_i\leq |\mtc{I}^v|-1    &  \forall d \in \mtc{D}, x^v \in \mtc{X}_d\setminus\mtc{K}_{d},\nonumber\\   
&\mu\geq \left(Q_{d^v}(x^v)-L_{d^v}\right)\left(\sum_{i\in\mtc{I}^v}x_i-\sum_{i\notin\mtc{I}^v}x_i\right)\nonumber\\
&\qquad -\left(Q_{d^v}(x^v)-L_{d^v}\right)\left(|\mtc{I}^v|-1\right)    \nonumber \\
&\qquad +L_{d^v}-U\left(1-\Ind_{d^v}(x)\right)  & \forall d \in \mtc{D}, x^v \in \mtc{X}_{d}\cap \mtc{K}_{d}. \nonumber
\end{align}

Given assumption A4 that $\mtc{X}=\{0,1\}^{n_1}$, the following result holds.
\begin{lemma}\label{prop:int:finite}
    There exists only finitely many cuts \eqref{eq:fc:binary} and \eqref{eq:ls:int:oc}.
\end{lemma}
\begin{proof}
We first observe that $\vert \mtc{X} \vert \leq 2^{n_1}$. 
We next note that there can only be one cut \eqref{eq:ls:int:oc} for each solution in $\mtc{X}$ that does not induce an infeasible second stage problem and one cut \eqref{eq:fc:binary} for each solution in $\mtc{X}$ that does. 
Thus, there can be at most $2^{n_1}$ cuts in total. 
\end{proof}

The analog of \Cref{prop:equivalent_reform} and \Cref{prop:convergence} for the case of second stage subproblems that are MILPs can be proven in a similar manner to those results.  This in turn proves that an L-Shaped method that iteratively adds cuts of the form \eqref{eq:fc:binary} and/or \eqref{eq:ls:int:oc} and does not remove them will ultimately produce an optimal solution to \eqref{eq:speu} as well as a proof of its optimality.

\subsection{Master problem formulations}\label{sec:benders:mpformulations}
An indicator function $\Ind_d(x)$ can be represented in \eqref{eq:rmp} by extending it to include binary variables $\delta_d$, $\forall d \in \mtc{D},$ as well as adding constraints that ensure $x\in \mtc{X}_d\iff \delta_d = 1 .$  This approach may be inconvenient in applications where the number of distributions $\vert\mtc{D}\vert$ is exponential in the size of the problem, as the ones we present in \Cref{sec:comp_setting}. We will now show that, under two sets of general conditions, it is possible to formulate \eqref{eq:rmp} and, consequently, the cuts without indicator functions $\Ind_d(x)$. The first set of conditions is when the $x$ variables can be continuous or integer. The second set of conditions is when the $x$ variables are binary.

Consider the case in which $\mtc{X}=\{x\in K^{n_1}\vert l^i\leq a_i^\top x\leq b^i,i=1,\ldots,m\}$ with $a_i\in\R^{n_1}$, $l^i,b^i\in\R$ and $K\in \{\R,\mathbb{Z}\}$.
In addition, assume that for $i=1,\ldots,m$, there exist constants $l^{ij}$ and $b^{ij}$ for $j=1,\ldots,J$ such that $l^{i1}=l^i$ and $b^{iJ}=b^i$ and $[l^{ij},b^{ij}] \cap [l^{ij'},b^{ij'}] = \emptyset, j,j' \in \{1,\ldots,J\}, j \neq j'.$
That is, each product $a_i^\top x$ may be in at most one out of $J$ disjoint boxes. 
Let variables $v_{ij}\in\{0,1\}, i = 1,\ldots, m$, $j=1,\ldots,J$ indicate in which box the product $a_i^\top x$ falls, that is
$$l^{ij}\leq a_i^\top x\leq b^{ij}\iff v_{ij}=1$$

Relatedly, we let the function $j(i,d):\{1,\ldots,m\}\times \mtc{D}\to \{1,\ldots,J\}$ return the index $j$ such that $\mtc{X}_d=\{l^{i,j(i,d)}\leq a_i^\top x\leq b^{i,j(i,d)}, i=1,\ldots,m\}$.
Then, given the variables $v_{ij}$ and the function $j(i,d)$,  the implication ``$x\in\mtc{X}_d\implies$'' can be enforced by rewriting, e.g., \eqref{eq:continuous:oc}, as follows

$$\mu\geq \sum_{s\in\mtc{S}_d}\pi_{sd}\rho^v_s(h_{sd}-T_{sd}x) - U(m-\sum_{i=1}^mv_{i,j(i,d)})$$

Here, if all the bounds implied by distribution $d$ apply, the last term becomes zero, otherwise it reduces to an upper bound.
This approach allows us to reformulate \eqref{eq:rmp} with $mJ$ binary variables $v$ instead of the $|\mtc{D}|=J^m$ binary variables $\delta_d$.
In a similar manner it is possible to rewrite cuts \eqref{eq:continuous:fc} and \eqref{eq:ls:int:oc}.

Consider now the case in which $\mtc{X}=\{0,1\}^{n_1}$.
Then each $\mtc{X}_d\subseteq\mtc{X}$ is itself a finite collection of points. Recall that the sets $\mtc{X}_d$ form a partition of $\mtc{X}$. Let $0=n^{(0)}<1<n^{(1)}<n^{(2)}<\cdots\leq n^{(T)}=n_1$ be integers that split the decision vector $x$ into $T\geq 1$ segments $(x_{n^{(0)}+1},\ldots,x_{n^{(1)}})$, $(x_{n^{(1)}+1},\ldots,x_{n^{(2)}})$, $\ldots$, $(x_{n^{(T-1)}+1},\ldots,x_{n^{(T)}})$. 
Consider a set $\mtc{K}_t$ of exhaustive conditions that may verify in segment $t$ (see the example in Appendix \ref{sec:app:example}). We have $|\mtc{D}|=\Pi_{t=1}^T|\mtc{K}_t|$ and we let $k(t,d)$ be the condition that must verify on segment $t$ for $x$ to be a member of  $\mtc{X}_d$. 
Then $\mtc{X}_d=\{x\in\{0,1\}^{n_1}|\Ind_{t,k(t,d)}(x)=1, t=1,\ldots,T\}$ where $\Ind_{t,k}(x)$ is an indicator function that takes value one if condition $k$ is verified on segment $t$, zero otherwise.
Then, let $v_{tk}\in\{0,1\}$ take value one if and only if condition $k\in\mtc{K}_t$ on segment $t$ holds. We can rewrite \eqref{eq:continuous:oc} as follows
$$\mu\geq \sum_{s\in\mtc{S}_d}\pi_{sd}\rho^v_s(h_{sd}-T_{sd}x) - U\left(T-\sum_{t=1}^Tv_{t,k(t,d)}\right)$$
In a similar manner it is possible to rewrite cuts \eqref{eq:continuous:fc} and \eqref{eq:ls:int:oc}. Thus, we use $\mtc{O}(T\times \max_{t}|\mtc{K}_t|)$ variables $v_{tk}$ instead of the $|\mtc{D}|=\Pi_{t=1}^T|\mtc{K}_t|$ variables $\delta_d$.

\subsection{A numerical example}
We next present a numerical example to illustrate the steps performed by the proposed method. 
Consider the following two-stage stochastic program
  \begin{align*}
    \min~ & x+\sum_{d=1}^2\Ind_d(x)\Ex_{\pr{P}_d}Q(x,\xi)\\
    &x\geq 0
  \end{align*}
  in which the random variable $\xi$ follows one of two discrete probability distributions, $\pr{P}_1$ and $\pr{P}_2$, depending on which of the sets $\mtc{X}_1,\mtc{X}_2$ contains the value of $x.$ These distributions are presented in \Cref{tbl:xi_dists} along with the definitions of the sets $\mtc{X}_1,\mtc{X}_2.$ Each distribution contains exactly two realizations. Relatedly, we  let $\Ind_1(x), \Ind_2(x)$ denote the characteristic functions of $\mtc{X}_1,\mtc{X}_2$, respectively. (Observe, that this problem could be trivially solved by solving two ordinary stochastic programs over $\mtc{X}_1$ and $\mtc{X}_2$, respectively. However, we believe its simplicity is useful in an illustrative example).

\begin{table}[htp]
\caption{Probability distributions for $\xi$.}
\label{tbl:xi_dists}

\centering
\begin{tabular}{cc|cc}
\toprule
\multicolumn{2}{c}{$\pr{P}_1: x \in \mtc{X}_1= [0.5,3]$} & \multicolumn{2}{c}{$\pr{P}_2: x \in \mtc{X}_2= [3.5,10]$}\\
\midrule
$\xi_{1,\cdot}$ & $\pi_{1,\cdot}$ & $\xi_{2,\cdot}$ & $\pi_{2,\cdot}$ \\
\hline
4 & 0.7 & 10 & 0.3 \\ 
12 & 0.3 & 18 & 0.7 \\
\bottomrule
\end{tabular}
\end{table}

Regarding the recourse function, we have
  $$Q(x,\xi)=\min_{y_1,y_2\geq 0} \left\{ y_1+2y_2 \vert y_1+y_2\geq 2+x, y_1\geq \xi-x\right\}$$
In terms of the general notation used in \eqref{eq:sp2s} we have $h = [2 \;\;\; \xi]^\top, T = [-1 \;\;\; 1]^\top$, $W= \begin{bmatrix}
1 & 1 \\
1 & 0 
\end{bmatrix}$ and $q=[1 \;\; 2]^\top$.
One can observe that $Q(x,\xi)$ is convex and piecewise linear. It also admits closed form solution. Namely, 
  $$
  Q(x,\xi)=
  \begin{cases}
    \xi-x,\qquad \text{~if~}x\leq (\xi-2)/2\\
    x+2, \qquad \text{~if~} x> (\xi-2)/2
  \end{cases}
  $$ 
The recourse function $Q(x)=\sum_{d=1}^2\Ind_d(x)\Ex_{\pr{P}_d}Q(x,\xi)$ is depicted in \Cref{fig:Qofx}. 

One can also observe that $Q(x,\xi)$ is a linear program and thus the proposed method will generate valid inequalities of the form $\mu\geq \sum_{s\in\mtc{S}_{d^v}}\pi_{sd^v}(\rho^v_s)^\top(h_{sd^v}-T_{sd^v}x) - U(1-\Ind_{d^v}(x))$ that were presented in \Cref{sec:benders:cont}.  The data elements $\rho_{s}^{v}$ represent extreme points associated with the feasible region of the dual of the second stage problem for scenario $s$ of distribution $d^{v}$, which is as below. 

\begin{align*}
Q(x,\xi) =  \max~&  (2 + x) \rho_{1} + (\xi - x) \rho_{2} \\
  &\rho_1  +  \rho_2 \leq 1 \\
 & \rho_2 \leq 2 \\
 &\rho_1,  \rho_2 \geq 0.
\end{align*}
The value $U$ in these valid inequalities is an upper bound on the greatest difference between the largest and smallest values of the recourse function across the potential distributions. In our case $U=12.5$ is a safe value.

To formulate the RMP we let $\delta_{i}$ denote whether or not $x \in \mtc{X}_{i}$. We observe that the implication of the domain of $x$ implied by the $\delta_{i}$ variables can be modeled with the constraint $0.5\delta_1+3.5\delta_2\leq x \leq 3\delta_1+10\delta_2$. We can thus formulate the initial RMP as follows
\begin{align}
    \min~ & x+\mu \label{rmp:obj}\\
         &0.5\delta_1+3.5\delta_2\leq x \leq 3\delta_1+10\delta_2, \label{rmp:c1}\\
         &\delta_1+\delta_2 =1, \label{rmp:c2}\\
         &\delta_1,\delta_2\in\{0,1\}, \label{rmp:c3}\\
         &\mu \geq 0. \label{rmp:c4}
  \end{align}

\noindent
\textbf{\underline{Iteration 1}}
At this iteration, the solution to RMP is $(x,\mu,\delta_1,\delta_2)=(0.5,0,1,0)$. As $x \in \mtc{X}_1,$ the method solves subproblems associated with distribution $\pr{P}_1.$ We observe that $Q(.5,\xi_{1,1})$ and $Q(.5,\xi_{1,2})$ have the same optimal solution to their duals, which is $(\rho_1,\rho_2)=(0,1).$ The cut generated is 
\begin{equation}\label{rmp:c5}\mu \geq 6.4 -x - 12.5(1-\delta_1)\end{equation}
which is depicted in \Cref{fig:cut1}.

\begin{figure}
   \begin{subfigure}[b]{.5\textwidth}
    \includegraphics[width=\textwidth]{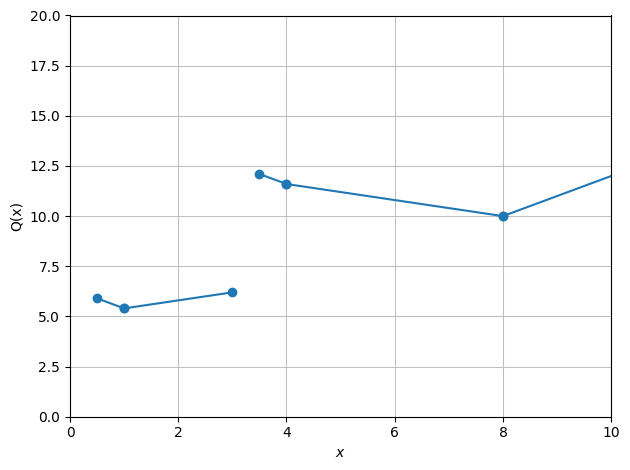}
    \caption{$Q(x)=\sum_{d=1}^2\Ind_d(x)\Ex_{\pr{P}_d}Q(x,\xi)$}
    \label{fig:Qofx}
    \end{subfigure}
\begin{subfigure}[b]{.5\textwidth}
    \includegraphics[width=\textwidth]{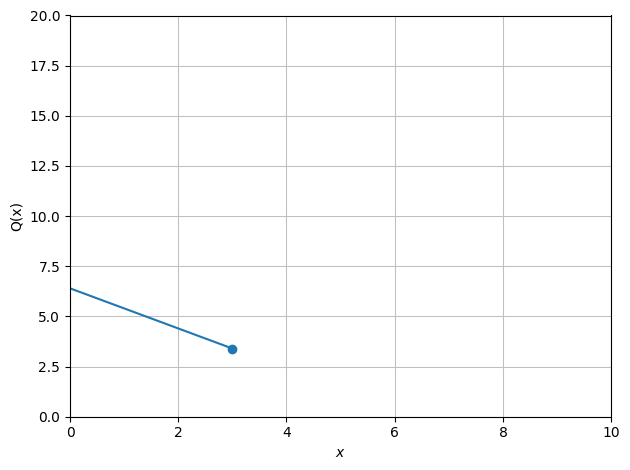}
    \caption{Optimality cuts after  first iteration}
    \label{fig:cut1}
\end{subfigure}
\begin{subfigure}[b]{.5\textwidth}
    \includegraphics[width=\textwidth]{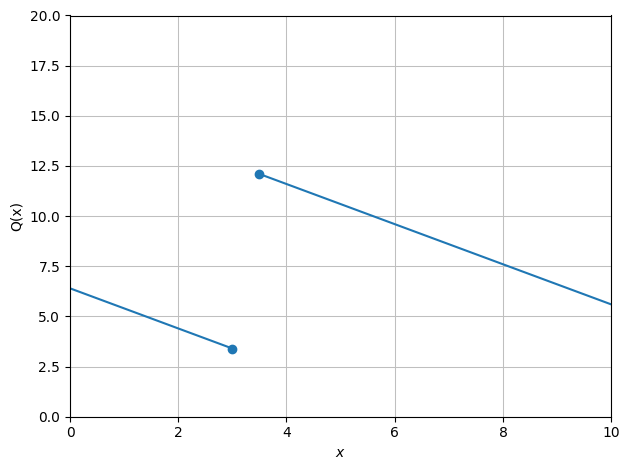}
    \caption{Optimality cuts after  second iteration}
    \label{fig:cut2}
\end{subfigure}
\begin{subfigure}[b]{.5\textwidth}
    \includegraphics[width=\textwidth]{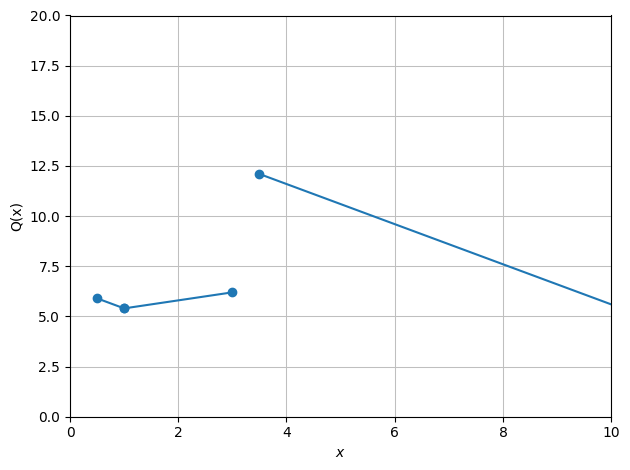}
    \caption{Optimality cuts after  third iteration}
    \label{fig:cut3}
 \end{subfigure}
 \caption{Recourse function and cuts added.}
 \label{fig:rec_cuts}
\end{figure}

\noindent
\textbf{\underline{Iteration 2:}}
At this iteration the method solves the RMP \eqref{rmp:obj}-\eqref{rmp:c4} along with the constraint \eqref{rmp:c5} to find the optimal solution $(x,\mu,\delta_1,\delta_2)=(3.5,0,0,1)$. As $x \in \mtc{X}_2$, the method solves subproblems associated with distribution $\pr{P}_2.$ We observe that $Q(3.5,\xi_{2,1})$ and $Q(3.5,\xi_{2,2})$ have the same optimal solution to their duals, which is $(\rho_1,\rho_2)=(0,1).$ The optimality cut generated is
\begin{equation}\label{eq:c6}
    \mu \geq 15.6 - x - 12.5(1-\delta_2)
\end{equation}
which is depicted in \Cref{fig:cut2}.

\noindent
\textbf{\underline{Iteration 3:}}
At this iteration the method solves the RMP \eqref{rmp:obj}-\eqref{rmp:c5} along with \eqref{eq:c6} to find the optimal solution  $(x,\mu,\delta_1,\delta_2)=(3,3.4,1,0)$. Given that $x \in \mtc{X}_1,$ the method solves the subproblems $Q(3,\xi_{1,1})$ and $Q(3,\xi_{1,2})$ associated with distribution $\pr{P}_1.$ The optimal dual solutions are $(\rho_1,\rho_2)=(1,0)$ for the first scenario and $(\rho_1,\rho_2)=(0,1)$ for the second. The optimality cut is
\begin{equation}
\mu \geq 5 + .4x - 12.5(1-\delta_1), \label{eq:c7}
\end{equation}
 which is depicted in \Cref{fig:cut3}.

\noindent
\textbf{\underline{Iteration 4:}}
At this iteration the method solves the RMP \eqref{rmp:obj}-\eqref{eq:c6} along with \eqref{eq:c7} to find the optimal solution  $(x,\mu,\delta_1,\delta_2)=(1,5.4,1,0)$. By evaluating the subproblems associated with distribution $\pr{P}_1$ the method can conclude that this solution is optimal.

In this example, the method converged having generated only one inequality to approximate $\Ex_{\pr{P}_2}Q(x,\xi)$. In other words, with only one optimality cut the method was able to rule out the possibility of $\mtc{X}_2$ containing the optimal solution. As we will see in \Cref{sec:comp_study}, this behavior is not unique to this example. Namely, we will see that the method often generates fewer than two inequalities per distribution.

\subsection{Distribution-independent valid inequalities}
\label{subsec:dist_ind_opt}
The proposed extension of the L-Shaped method requires adding to the master problem distribution-specific inequalities.
We next present a class of distribution-independent valid inequalities. These are, in general, not tight and do not guarantee convergence when used alone.
However, they may provide non-trivial lower bounds that facilitate convergence.

Recall that the recourse function of \eqref{eq:speu} is defined as $$Q(x)=\sum_{d\in\mtc{D}}\Ind_d(x)\Ex_{\pr{P}_d}Q(x,\xi_{sd}).$$
The following proposition paves the way to a set of distribution-independent valid inequalities.

\begin{proposition}\label{prop:lpbound}
  For all $x\in\mtc{X}$ it holds that
  $$Q(x)\geq \sum_{d\in\mtc{D}}\sum_{s\in\mtc{S}_d}\pi_{sd}L^{d}(x,\xi_{sd})$$
  where, for given $d$ and $s\in \mtc{S}_d$, $L^{d}(x,\xi_{sd})$ is the linear programming relaxation to the following integer programming problem
  \begin{align*}Q^{d}(x,\xi_{sd})=&\min_{y\in\mtc{Y},\tau\geq 0,\delta\in\{0,1\}}q_{sd}^\top \tau\\
                                  & W_{sd}y=h_{sd}-T_{sd}x\\
                                  &\delta =1\iff x\in\mtc{X}_d\\
                                  &\tau-y \leq 0\\
                                  &\tau-Y\delta\leq 0\\
                                  &\tau-y-Y\delta\geq -Y
  \end{align*}
  where $Y\in \R^{n_2}$ is such that $Y\geq y$ (element-wise) for all $y\in\mtc{Y}$.
\end{proposition}
  \begin{proof}

    Under assumption A1 we have 
    $$Q(x)=\sum_{d\in\mtc{D}}\Ind_d(x)\sum_{s\in\mtc{S}_d}\pi_{sd}Q(x,\xi_{sd})=\sum_{d\in\mtc{D}}\Ind_d(x)\sum_{s\in\mtc{S}_d}\pi_{sd}\min_{y\in\mtc{Y}}\{q_{sd}^\top y\vert W_{sd}y=h_{sd}-T_{sd}x\}$$
    Having $\Ind_d(x)\geq 0$ for all $x$ and $d\in\mtc{D}$, by homogeneity we can move it inside the minimization problem obtaining
    \begin{align}
      Q(x)=\sum_{d\in\mtc{D}}\sum_{s\in\mtc{S}_d}\pi_{sd}\min_{y\in\mtc{Y}}\{q_{sd}^\top y\Ind_d(x)\vert W_{sd}y=h_{sd}-T_{sd}x\}
    \end{align}
    Let us refer to the new minimization problem as $Q^{d}(x,\xi_{sd})$.
    Inside each $Q^{d}(x,\xi_{sd})$ the $\Ind_d(x)$ can be implemented by a binary variable $\delta$ such that $\delta=1\iff x\in\mtc{X}_d$.
    We obtain
    \begin{align}Q^{d}(x,\xi_{sd})=\min_{y\in\mtc{Y},\delta\in\{0,1\}}&\left\{q_{sd}^\top y\delta\left\vert W_{sd}y=h_{sd}-T_{sd}x, \delta =1\iff x\in\mtc{X}_d\right.\right\}
    \end{align}
    Observe that $\delta$ is determined by $x$.

    The bilinear terms in the objective can be replaced by their McCormick envelopes \cite{Mcc76}. We obtain
    \begin{align}Q^{d}(x,\xi_{sd})=&\min_{y\in\mtc{Y},\tau\geq 0,\delta\in\{0,1\}}q_{sd}^\top \tau\\
                                                   & W_{sd}y=h_{sd}-T_{sd}x\\
                                                   &\delta =1\iff x\in\mtc{X}_d\\
                                                   &\tau-y \leq 0\\
                                                   &\tau-Y\delta\leq 0\\
                                                   &\tau-y-Y\delta\geq -Y
    \end{align}
    It follows that, for all $x\in\mtc{X}$
    $$Q(x)=\sum_{d\in\mtc{D}}\sum_{s\in\mtc{S}_d}\pi_{sd}Q^{d}(x,\xi_{sd})\geq \sum_{d\in\mtc{D}}\sum_{s\in\mtc{S}_d}\pi_{sd}L^{d}(x,\xi_{sd})$$
    where the inequality holds because $Q^{d}(x,\xi_{sd})\geq L^{d}(x,\xi_{sd})$ for all $x$, $d$ and $s$. 
  \end{proof}

Next, let $\nabla^v_{sd}$ be a subgradient to $L^d(x,\xi_{sd})$ at $x^v$. \Cref{prop:lpbound} and convexity of $L^d(x,\xi_{sd})$ in $x$ suggest adding cuts 
\begin{equation}\label{eq:ext:vi1}
  \mu\geq \sum_{d\in\mtc{D}}\sum_{s\in\mtc{S}_d}\pi_{sd}L^d(x^v,\xi_{sd})+(\nabla^v_{sd})^\top(x-x^v)
\end{equation}
generated from $L^{d}(\cdot,\cdot)$ upon visiting solutions $x^v$ to \eqref{eq:rmp}.
These cuts support the recourse function from below and create a convex lower envelop to $Q(x).$ 

Observe that cuts \eqref{eq:ext:vi1}, while general, require solving one linear program for each $d$ and for each $s$.
Under specific assumptions, we can generate distribution-independent inequalities by solving one subproblem per distribution.
\begin{proposition}\label{prop:evproblem}
  Assume $Q(x,\xi)$, hence $L^d(x,\xi)$, are convex in $\xi$. Then, for all $x\in\mtc{X}$
$$Q(x)\geq \sum_{d\in\mtc{D}}L^{d}(x,\Ex_{\pr{P}_d}[\xi])$$
\end{proposition}  
\begin{proof}
  The result follows immediately from Jensen's inequality.
  \end{proof}  

This result allows us to solve only one subproblem per distribution instead of $\vert \mtc{S}_d\vert$.
Hence, \eqref{eq:ext:vi1} reduces to
\begin{equation}\label{eq:ext:vi1conv}
  \mu\geq \sum_{d\in\mtc{D}}L^d(x^v,\Ex_{\pr{P}_d}[\xi])+(\nabla^v_{d})^\top(x-x^v)
\end{equation}
where $\nabla^v_{d}$ is a subgradient to $L^d(x,\Ex_{\pr{P}_d}[\xi])$ at $x^v$.

Another opportunity for generating distribution-independent cuts involves examining the scenarios themselves to see if they meet certain criteria. 
\begin{proposition}\label{prop:lowerenvelop}
  Assume there exists $\hat{\xi}=(\hat{q},\hat{W},\hat{T},\hat{h})$ such that $Q(x,\hat{\xi})\leq \Ex_{\pr{P}_d}Q(x,\xi)$ for all $x\in\mtc{X}$ and $d\in\mtc{D}$, where $Q(x,\xi)$ is a linear program.
  Let $\rho^v$ be an optimal dual solution to $Q(x^v,\hat{\xi})$ for some $x^v\in\mtc{X}$. Then
  $$\Ex_{\pr{P}_{d^l}}[Q(x^l,\xi)]\geq (\rho^v)^\top(\hat{h}-\hat{T}x^l)$$
  for all $x^l\in \mtc{X}$.
  \end{proposition}
  \begin{proof}
    Consider $x^l\in\mtc{X}$ which yields $d^l$.
    \begin{align*}
      \Ex_{\pr{P}_{d^l}}[Q(x^l,\xi)]=& \sum_{s\in\mtc{S}_{d^l}}\pi_{sd^l}Q(x^l,\xi_{sd^l})\\
      \geq &Q(x^l,\hat{\xi})=(\rho^l)^\top(\hat{h}-\hat{T}x^l) \geq (\rho^v)^\top(\hat{h}-\hat{T}x^l)
    \end{align*}
    where $\rho^l$ and $\rho^v$ are, respectively, optimal and feasible to the dual of $Q(x^l,\hat{\xi})$. 
  \end{proof}  

\Cref{prop:lowerenvelop} justifies adding valid inequality
\begin{equation}\label{eq:ext:vi2}\mu \geq (\rho^v)^\top(\hat{h}-\hat{T}x)
\end{equation}
from $x^v$, if the optimality test fails.
Compared to \Cref{prop:lpbound}, \Cref{prop:lowerenvelop} requires solving only one additional linear program at every iteration.
However, it requires the existence and identification of a scenario $\hat{\xi}$, which might not be trivial. Next, we describe a special case where this task is simplified.
\begin{definition}
  For all $x\in \mtc{X}$ the function $Q(x,\xi):\R^N\to \R$ is said to be \textit{monotone} non-decreasing/non-increasing if for all $i=1,\ldots,N$, given and fixed $\xi^j=\hat{\xi}^j\in \R$ for $j\in\{1,\ldots,N\}\setminus\{i\}$, the functions $Q(x,(\hat{\xi}^1,\ldots,\hat{\xi}^{i-1},\xi^i,\hat{\xi}^{i+1},\ldots,\hat{\xi}^{N})):\R\to \R$ are monotone non-decreasing/non-increasing.
\end{definition}

\begin{proposition}\label{prop:min_ev}
  Assume $Q(x,\xi)$ is convex and monotone non-decreasing in $\xi$ for all $x\in\mtc{X}$.
  Consider $\hat{\xi}=\min_{d\in\mtc{D}}\Ex_{\pr{P}_d}[\xi]$ where the minimum is taken component-wise.
  Then $\hat{\xi}$ satisfies the definition in \Cref{prop:lowerenvelop}.
\end{proposition}
  \begin{proof}
    Let $x^v\in \mtc{X}$ be arbitrary and $d^v$ the corresponding distribution.
    Then
    $$\Ex_{\pr{P}_d}[Q(x^v,\xi)]\geq Q(x^v,\Ex_{\pr{P}_d}[\xi])\geq Q(x^v,\hat{\xi})$$
    Where the first inequality holds by by convexity and Jensen's inequality and the second by monotonicity and by definition of $\hat{\xi}$.
  \end{proof}

\section{Application of proposed method to production planning}
\label{sec:comp_setting}

To assess the computational effectiveness of the proposed method, we apply it to three variants of a production planning problem under both endogenous and exogeneous uncertainty. While all three variants possess the properties required by the proposed framework, each presents a different computational challenge with respect to the solution of the second stage subproblem. We note that many of the concepts and notation are common across the three variants. We introduce these when introducing the first variant. When presenting the second and third variants we only discuss how they differ from the first.

\subsection{Production Planning Problem - Variant 1}
\label{sec:model}
 A manufacturer can produce a product at one or more facilities $\mathcal{F}$. 
 When there is demand for a unit, it can be sold at price $P.$ When there is not, it can be sold at price $O < P$.
 The manufacturer seeks to determine how much to make at each facility $f \in \mathcal{F}$, given a per unit cost of production, $C_{f}.$
 
 For each facility $f$, the manufacturer chooses a production level $l$, from a set $\mathcal{L}_{f}$, that defines a production volume range $[L_{fl},U_{fl}]$. 
 We presume that the set $\mathcal{L}_{f}$ includes a level $l$ such that $L_{fl}=0$. Furthermore, we assume that for all $f\in\mtc{F}$ we have $[L_{f,l_1},U_{f,l_1}]\cap [L_{f,l_2},U_{f,l_2}]=\emptyset$ for all $l_1,l_2\in\mtc{L}_f$.

When planning production, the yield at each facility, and the cumulative demand of finished goods for the product are uncertain. 
We let $\rv{\xi}=(\mathbf{Y}_{f_1},\ldots,\mathbf{Y}_{f_{|\mathcal{F}|}},\mathbf{D})$ denote an $\R^{|\mtc{F}|+1}$ random vector representing the uncertain yield at the different facilities, $\mathbf{Y}_{f}$, and the cumulative demand, $\mathbf{D}$. 
The yield at a given facility can also depend on the production levels chosen.

We let $\mtc{D}$ denote the set of joint distributions of yields and demand. However, as uncertainty in demand is exogenous, its marginal distribution  is identical for all distributions in $\mtc{D}.$ 
For each facility $f$, and distribution $\mtc{D}$, there is a map $l(f,d):\mtc{F}\times \mtc{D}\to \cup_{f\in\mtc{F}}\mtc{L}_f$ that indicates the production level at facility $f$ required for distribution $d \in \mtc{D}$. 
It is necessary that the production level indicated by the map $l(f,d)$ is chosen for each facility $f$ for probability distribution $d\in\mtc{D}$ to accurately represent uncertainty in yields. Observe that the resulting number of distributions is $\vert \mathcal{D} \vert =  \Pi_{f\in\mtc{F}}\vert  \mathcal{L}_f \vert$.

We further assume that each probability distribution is either discrete or can be adequately represented by a finite number of scenarios. 
We let $\mtc{S}_{d}$ represent the set of scenarios in distribution $d \in \mtc{D}$. 
Associated with scenario $s \in \mtc{S}_{d}$ for distribution $d \in \mtc{D}$ we let $\pi_{sd}$ denote the probability that it occurs. Regarding yields at facilities, we let $Y_{fsd}$ represent the realization of the yield at facility $f$ under scenario $s\in\mtc{S}_{pd}$ of distribution $d\in\mtc{D}$. We let $D_{sd}$ represent the realization of the demand  under scenario $s\in\mtc{S}_{d}$ of distribution $d\in\mtc{D}$.

To model the problem under consideration, we let $x_{f}\geq 0$ represent the amount of raw materials, quoted in terms of finished goods, allocated to facility $f \in \mathcal{F}$. We let binary variable $y_{fl}$ equal $1$ if production level $l \in \mathcal{L}_{fp}$ is chosen at facility $f \in \mathcal{F}$ and $0$ otherwise.  The binary decision variable $\delta_{d}$ represents whether distribution $d \in \mathcal{D}$ should be used to describe yield uncertainties. Decision variables $x:=(x_{f})_{f\in\mathcal{F}}$, $y:=(y_{fl})_{f\in\mathcal{F},l \in \mathcal{L}_{f}}$ and $\delta:=(\delta_{d})_{d\in\mathcal{D}}$ are first stage decision variables.

\begin{subequations}\label{eq:ppp1}
\begin{align}
\label{eq:model1:obj}\min~&\sum_{f\in\mathcal{F}}C_{f}x_{f}-\sum_{d\in\mathcal{D}}\delta_{d}\Bigg[\sum_{s\in\mathcal{S}_{d}}\pi_{sd}Q(x,\xi_{sd})\Bigg] \\
\label{eq:model1:c1}\text{s.t. ~} &   \sum_{l\in\mathcal{L}_{f}}y_{fl} =1&\forall f\in\mathcal{F},  \\
\label{eq:model1:c2} &   \sum_{l\in\mathcal{L}_{f}}L_{fl}y_{fl} \leq x_{f}\leq \sum_{l\in\mathcal{L}_{f}}U_{fl}y_{fl}&\forall f\in\mathcal{F}, \\ 
\label{eq:model1:c4} &\sum_{f\in\mathcal{F}}y_{f,l(f,d)}\geq |\mathcal{F}|\delta_{d} &\forall d\in \mathcal{D}, \\
\label{eq:model1:fs_dvx} & x_{f} \geq 0  &\forall f\in\mathcal{F}, \\ 
\label{eq:model1:fs_dvy} & y_{fl} \in \{0,1\} &\forall f\in\mathcal{F}\, l \in L_{f}, \\
\label{eq:model1:fs_dvd} & \delta_{d} \in \{0,1\} &\forall d \in \mathcal{D}.
\end{align}
\end{subequations}

The objective \eqref{eq:model1:obj} seeks to maximize expected profit. Here, $Q(x,\xi_{sd})$ represents the revenue earned by selling inventory under realization $\xi_{sd}$ of yield and demand. 
Constraints \eqref{eq:model1:c1} ensure that a production level is chosen for each product at each facility. 
Constraints \eqref{eq:model1:c2} ensure that the amount allocated to a facility falls within the bounds defined by the chosen production level. Constraints \eqref{eq:model1:c4} ensure that the choice of production levels enforces the appropriate distribution. 

Once yield and demand information have materialized through a scenario $s\in\mathcal{S}_{d}$ for the distribution $d\in\mathcal{D}$ indicated by $\delta_{d}=1$, the decision maker makes second stage decisions concerning sales. 
The continuous variable $w$ defines the amount of product that is sold at ``full'' price $P$ while the continuous variable $o$ represents the amount of inventory that is sold at the discounted price, $O.$ For $d\in\mtc{D}$ and $s\in\mtc{S}_d$ the second stage decision problem is as follows.

\begin{align}    \label{eq:ppp1:sp}
Q(x,\xi_{sd}) = \max_{w, o \geq 0}\left\{P w+Oo\left\vert w + o= \sum_{f\in\mathcal{F}}Y_{fsd}x_{f}, w\leq D_{sd}\right.\right\}&
\end{align}

The constraints limit the amount of inventory that is sold by how much is available given the yields. Furthermore, the amount of inventory that is sold at full price is limited by by the demand realization.

\subsection{Production Planning Problem with Per-unit Transportation Costs - Variant 2}
\label{subsec:pp_pu_trans_costs}
This problem is like the one just discussed, except that demand for finished goods occurs in different locations. 
Thus, for this variant, the first stage decisions are the same as in problem \eqref{eq:ppp1}. However, the second stage problem is different as we must determine to which customer locations finished goods should be transported to meet demand. 

We let $\mtc{Q}$ denote the set of locations wherein demand can occur. 
Hence, $\rv{\xi}$ is now an $\R^{|\mtc{F}|+|\mtc{Q}|}$ random vector.
We let $D_{qsd}$ represent the demand realization for finished goods in location $q \in \mtc{Q}$ given distribution $d \in \mtc{D}$ and scenario $s \in \mtc{S}_{d}$. We also let $H_{fq}, f \in \mtc{F}, q \in \mtc{Q}$ represent the per-unit cost associated with transporting units from facility $f$ to demand location $q.$ 

We let the continuous decision variable $w_{fq}$ denote the amount of finished goods transported from facility $f \in \mtc{F}$ to demand location $q \in \mtc{Q}$ and $o_f$ the amount of inventory left at facility $f$. The second stage problem is as follows.
\begin{subequations}
    \label{eq:ppp2:sp}
\begin{align}
Q(x,\xi_{sd})=\max~&\sum_{f\in\mathcal{F}}\sum_{q\in\mtc{Q}}(P-H_{fq})w_{fq}+\sum_{f\in\mathcal{F}}Oo_{f}\\
\label{eq:ppp2:c5}  &\sum_{q\in\mtc{Q}}w_{fq}+o_{f}= Y_{fsd}x_{f} & \forall f\in\mathcal{F},\\
\label{eq:ppp2:c6}  &\sum_{f\in\mathcal{F}}w_{fq}\leq D_{qsd}&\forall  q\in\mtc{Q},\\
\label{eq:ppp2:r1} &  o_{f} \geq 0 & \forall f\in\mathcal{F}, \\
\label{eq:ppp2:r2} & w_{fq} \geq 0 & \forall f \in \mathcal{F},q\in\mtc{Q}.
\end{align}
\end{subequations}

\subsection{Production Planning Problem with Batches and Per-vehicle Transportation Costs - Variant 3}
\label{subsec:pp_pv_trans_costs}
Finally, we consider a variant of the problem that differs in two ways from those previously presented. First, production decisions in the first stage are made in terms of batches as opposed to continuous amounts. Second, while we again presume demand for finished goods occurs in different locations, we now presume that transportation costs are paid on a per-vehicle basis. 

We presume a set of potential production amounts, $\mtc{B}_{f}$ for facility $f \in \mtc{F},$ at most one of which must be chosen. Associated with batch $b \in \mtc{B}_{f}$ is a batch size $Q_{bf}$. We let the binary variable $x_{fb} \in \{0,1\}$ denote whether batch size $b \in B_{f}$ is chosen for facility $f.$ As in the previous variants, we assume probability distributions are determined by production levels.  
\begin{subequations}\label{eq:ppp3}
\begin{align}
\label{eq:ppp3:obj} \min~&\sum_{f\in\mathcal{F}}\sum_{b \in \mtc{B}_{f}}C_{f}Q_{bf}x_{fb}- \sum_{d\in\mathcal{D}}\delta_{d}\Bigg[\sum_{s\in\mathcal{S}_{d}}\pi_{sd}Q(x,\xi_{sd})\Bigg]  \\
\label{eq:ppp3:c0}\text{s.t.} & \sum_{b \in \mtc{B}_{f}}x_{fb} \leq 1 & \forall f \in \mathcal{F}, \\
\label{eq:ppp3:c1} &   \sum_{l\in\mathcal{L}_{f}}y_{fl} =1 & \forall f\in\mathcal{F},  \\
\label{eq:ppp3:c2} &   \sum_{l\in\mathcal{L}_{f}}L_{fl}y_{fl} \leq \sum_{b \in \mtc{B}_{f}}Q_{bf}x_{fb}\leq \sum_{l\in\mathcal{L}_{f}}U_{fl}y_{fl} & \forall f\in\mathcal{F} \\ 
\label{eq:ppp3:c3}    & \sum_{d\in\mathcal{D}}\delta_{p} = 1 &  \\
\label{eq:ppp3:c4} &\sum_{f\in\mathcal{F}}y_{f,l(f,d)}\geq |\mathcal{F}|\delta_{d} &\forall d\in \mathcal{D} \\
\label{eq:ppp3:r1} & x_{fb} \in \{0,1\}  &\forall f\in\mathcal{F}, b \in B_{f}\\ 
\label{eq:ppp3:r2} & y_{fl} \in \{0,1\} &\forall f\in\mathcal{F}, l \in L_{f} \\
\label{eq:ppp3:r3} & \delta_{d} \in \{0,1\} &\forall d \in \mathcal{D} 
\end{align}
\end{subequations}
Let $K$ be the capacity of the vehicles and $G_{fq}$ the cost of a vehicle traveling from facility $f \in \mtc{F}$ to customer location $q \in \mtc{Q}$. Thus, we define integer decision variables $v_{fq}$ that model the number of vehicles that travel from facility $f \in \mtc{F}$ to location $q \in \mtc{Q}$. The second stage problem is as follows.
\begin{subequations}\label{eq:ppp3:sp}
\begin{flalign}
Q(x,\xi_{sd}) = \max~&\sum_{f\in\mathcal{F}}\left(\sum_{q\in\mtc{Q}}(Pw_{fq} - G_{fq}v_{fq}) +Oo_{f}\right) & \\
\label{eq:ppp3:sp:c0}\text{s.t.~} &\sum_{q\in\mtc{Q}}w_{fq}+o_{f}= Y_{fsd}\sum_{b \in \mtc{B}_{f}}Q_{bf}x_{fb} & \forall f\in\mathcal{F}, \\
\label{eq:ppp3:sp:c6}  &\sum_{f\in\mathcal{F}}w_{fq}\leq D_{qsd}&\forall  q\in\mtc{Q},\\
\label{eq:ppp3:sp:c7} & w_{fq} \leq Kv_{fq} & \forall f \in \mtc{F}, q \in \mtc{Q}, \\ 
\label{eq:ppp3:sp:r1} &  o_{f} \geq 0 & \forall f\in\mathcal{F}, \\
\label{eq:ppp3:sp:r2} & w_{fq} \geq 0 & \forall f \in \mathcal{F}, q \in \mathcal{Q},\\
\label{eq:ppp3:sp:r3} & v_{fq} \in \mathbb{Z}^+ & \forall f \in \mathcal{F}, q \in \mathcal{Q}.
\end{flalign}
\end{subequations}

\section{Computational study}
\label{sec:comp_study}

To assess the effectiveness of the proposed, general, method we use it to solve instances of each variant of the production planning problem presented in \Cref{sec:comp_setting}. We compare the performance of the method with an off-the-shelf optimization solver solving instances of the extensive linearized form of the three variants of the problem. The linearized versions are obtained by linearizing the bilinear terms in the objective functions using \cite{Mcc76,AlkF83}. In Appendix \ref{app:lin} we present the linearized extensive form of the first variant. Linearized extensive forms of the other variants are similar.

We begin this section presenting the instances used in the experiments that form the basis of our computational analyses and the computational environment in which the experiments were run. We finish the section by discussing and analyzing the results of those experiments.

\subsection{Instances}
\label{subsec:instances}

All instances underlying our computational study are randomly generated, albeit with a process inspired by a motivating industrial application. Thus, we will next describe the parameter values used when generating those instances. We note that the parameter values defining an instance of the second or third variants include the parameter values required by the first. We also note that all instances are available upon request. 

We recall that there are two primary parameters, $\vert \mathcal{F} \vert$ and $\vert \mathcal{L}_{f} \vert$, that (partially) define a set of instances for all three variants. Regarding the first parameter, for the first variant we consider instances that contain between two and seven facilities. For the second and third variants we consider instances that contain between two and five facilities. Regarding the second parameter, each facility has the same number of production levels, which ranges from two to five. For the second and third variants there is another parameter, which is the number of customer locations, $\vert \mathcal{Q} \vert.$ In all instances of variants two or three we considered five such locations. Finally, for the third variant another parameter is the number of production batches, $\vert \mathcal{B} \vert.$ In all instances of variant three we considered five batches. 

Regarding the economics underlying the instances, the cost $C_{f}$ of manufacturing in facility $f \in \mathcal{F}$ was randomly drawn from the range $[60,80].$ The price $P$ at which product is sold was randomly drawn from the range $[125,185]$. As such, the expected gross margin on products sold is 54\%.  Finally, the revenue $O$ recovered from left-over inventory was randomly drawn from the range $[15,40].$ Note this ensures that in each instance there is a loss on products manufactured but not sold at full price. Turning to transportation costs, we randomly generated distances in miles between between each facility and each customer location based on a uniform distribution over the interval [75,300].  Per-unit costs, $H_{fq}, f \in \mathcal{F}, q \in \mathcal{Q},$ were computed based on those distances and a cost of \$0.10 per mile. Per-vehicle costs, $G_{fq}, f \in \mathcal{F}, q \in \mathcal{Q},$ were computed similarly albeit based on a per-mile cost of \$1.55.

We next discuss the scenario generation process. We first recall that associated with each instance of any of the three variants there are $\vert \mathcal{D} \vert =  \Pi_{f\in\mtc{F}}\vert  \mathcal{L}_f \vert$ distributions. The number of distributions in our instances for variant 1  ranged from $2^2=4$ to $5^7=78,125$. For variants 2 and 3, as those instances consist of at most five facilities, the number of distributions ranged from $2^2=4$ to $5^5=3,125.$ 

For the uncertainty in production yields, we consider a single joint distribution for each combination of production levels at different facilities.
This joint distribution is composed of marginal normal distributions truncated at $0.25$ and $1.0$, with different parameters at different facilities as well as at different production levels within the same facility. Generally speaking, we considered individual yield distributions such that for a given mean, the larger the production level the smaller the standard deviation in production yield.
Demands were also presumed to follow a normal distribution. In the case of variants 2 and 3, in which demands occur at customer locations, we presumed the demand distributions at different customer locations were independent. However, the mean of the distribution varies across customer locations.

Scenarios were generated by random sampling from the given joint distribution. We considered $S =$ $5$, $10$, $15$, $\ldots$, $45$, $50$ scenarios per distribution. As a result, an instance consists of $S\Pi_{f\in\mtc{F}}\vert\mathcal{L}_f \vert$  scenarios. %
\Cref{tab:size} reports statistics on the size of the extensive linearized formulations of the three variants for the smallest and largest number of distributions considered.

\begin{table}\caption{Average sizes of the linearized extensive forms by variant for the smallest and largest number of distributions.}\label{tab:size}

    \centering
    \begin{tabular}{cccccccc}
    \toprule
             &  &\multicolumn{2}{c}{\# variables}  & \multicolumn{2}{c}{\# int. variables}   & \multicolumn{2}{c}{\# constraints}  \\
        Variant & $\vert\mtc{D}\vert$ & min & max & min & max & min & max \\
        \midrule
         V1& 4 &  410 &	40,010 &	8 &	8 &	813 	&80,013\\
         V1& 78,125 & 78,90,667 &	781,328,167 &	78,160 &	78,160 &	15,703,154 	&1,562,578,154 \\
         V2& 4 & 2,414 &	240,014 &	12 &	12 	&4,313 &	430,013\\
         V2& 3,125 & 4,703,155 &	468,765,655 &	15,650 &	15,650 &	7,815,646 	&781,253,146 \\
         V3& 4 & 4,422 &	440,022 &	2,022 &	200,022 &	8,313 &	830,013\\
         V3& 3,125& 8,609,425 &	859,390,675 &	3,921,925 	&390,640,675 &	15,628,146 &	1,562,503,146\\
         \bottomrule
    \end{tabular}
\end{table}

For all three variants we define a class of instances by the triplet $(\vert \mathcal{F} \vert, \vert \mathcal{L} \vert, S)$, wherein $\vert \mathcal{L}_{f} \vert = \vert \mathcal{L} \vert \;\; \forall f \in \mathcal{F}.$ We note that for each variant we generate five instances for each class. In total, we consider 880 instances of the first variant, 800 instances of the second, and 800 instances of the third.

\subsection{Computational environment}
\label{subsec:comp_environment}
All code supporting both the proposed method presented in \Cref{sec:benders} as well as solving the extensive linearized versions is written in Java.
All optimization models solved by that code were done so using CPLEX 20.1 \cite{Cplex}.
The proposed method was implemented in a branch-and-benders-cut-type fashion. Namely, the relaxed master problem is solved by a branch-and-bound-based method and inequalities generated at nodes of the tree at which integer solutions are found.
We also note that when solving instances with integer second stage problems (e.g. \eqref{eq:sp2s:integer}) we also solve the linear programming relaxation of the subproblem in order to generate duality-based cuts \eqref{eq:continuous:oc}.  
We note that the three problem variants considered have relatively complete recourse and thus feasibility cuts need not be generated. 

All tests were performed on computing nodes equipped with 40 CPUs and 188 GB memory.
However, no parallelization techniques were used when executing our L-Shaped algorithm for any of the three variants.
Each of those experiments was run using a single thread. However, when solving extensive linerized forms, CPLEX was configured to use its default deterministic parallel search strategy for the branch-and-bound method, which used up to $32$ threads.
In all experiments we used an optimality tolerance of $0.0001$ and a time limit of $1,800$ seconds.

\subsection{Performance of the L-shaped method} \label{subsec:perf_benders}

We compare the performance of two methods. The first is the proposed extension of the L-shaped method, which we label \textbf{LS} in the subsequent tables. We note that we initially report results based on experiments in which the distribution-independent valid inequalities presented in \Cref{subsec:dist_ind_opt} were not used. We end this section with a study of their impact.  The second method involves solving the extensive linearized forms with CPLEX, which we label \textbf{CPLEX} in the subsequent tables. To compare the computational effectiveness of the two methods we focus on four statistics. The first is the percentage of instances that method could solve to the desired tolerance and within the given time limit. The second is the average time, in seconds, that method required to solve those instances. The third is the average optimality gap reported by the method at termination for the instances it could not solve, but for which it completed without exceeding the memory limit. The fourth is the percentage of instances for which the method terminated because it exceeded the memory alotted.  We report these statistics, by variant, in \Cref{tbl:comp_bench}.

\begin{table}[htp]
  \centering
  \caption{Comparison of methods on three variants. Optimality gaps are computed as $100\cdot (|\texttt{best\_primal}-\texttt{best\_dual}|)/|\texttt{best\_dual}|$. }
  \label{tbl:comp_bench}  
  \begin{tabular}{|c|c|c|c|c|c|}
    \toprule
    \multirow{3}{*}{Variant 1} & \% of instances  	& Average time  & \% unsolved 	& \% unsolved & Avg. gap  \\
                               &	solved		& to solve (sec.) &	exceeded	& did not exceed & unsolved \\
                               &				&			  & memory	& memory		   &			\\
    \midrule
    \textbf{LS} & 100.00\% &  0.04  & N/A & N/A & N/A \\
    \textbf{CPLEX} & 90.57\% &  12.48  & N/A & N/A & 9.43\% \\
    \midrule
    \multirow{3}{*}{Variant 2} & \% of instances  	& Average time  & \% unsolved 	& \% unsolved & Avg. gap  \\
                               &	solved		& to solve (sec.) &	exceeded	& did not exceed & unsolved \\
                               &				&			  & memory	& memory		   &			\\\midrule
    \textbf{LS} & 96.25\% 		&  95.86  		& N/A		& 3.75\%		   & 95.62\%  \\
    \textbf{CPLEX} & 31.75\% 	&  462.56  	& 33.00\%		& 35.25\%		   &	98.59\%  \\
    \midrule
    \multirow{3}{*}{Variant 3} & \% of instances  	& Average time  & \% unsolved 	& \% unsolved & Avg. gap  \\
                               &	solved		& to solve (sec.) &	exceeded	& did not exceed & unsolved \\
                               &				&			  & memory	& memory		   &			\\\midrule
    \textbf{LS} & 100.00\% 		&  355.81  	& N/A 		& N/A		   & N/A  \\
    \textbf{CPLEX} & 49.88\% 	&  332.67  	& 21.50\%		& 28.63\% 	   & 30.69\% \\
    \bottomrule
  \end{tabular}
\end{table}

We see that while solving the deterministic equivalent with CPLEX was fairly effective for the first variant, it solved less than half the instances of the second and third  variants. Further, CPLEX exceeded the alotted memory when solving many instances. Finally, when CPLEX did not exceed the alotted memory, it reported optimality gaps over 30\% for the instances it did not solve. 

 The proposed method was able to solve all instances of variants 1 and 3 and 96.25\% of instances of variant 2. For the instances the method could solve, it required, on average, less than six minutes to do so. We further note that all the instances of variant 2 that the proposed method did not solve were instances involving 3,125 distributions.  Unlike CPLEX, the proposed method rarely required more than $1$ megabyte of memory during its execution.  We conclude from these results that the proposed method is superior to the benchmark, at least for this class of problem. Thus, we next focus on analyzing the performance of the proposed method.

We next analyze three more statistics regarding the performance of the method. The first is the number of nodes in the branch-and-bound tree the method explored. The second is the number of optimality cuts the method generated. The third is the time the method spent generating cuts. We report averages of these statistics, by variant, in  \Cref{tbl:perf_variant}. We note we only report these statistics for instances the method was able to solve to optimality. 
\begin{table}[htp]
\centering
\caption{Performance statistics by variant.}
\label{tbl:perf_variant}
\begin{tabular}{cccc}
\toprule
Variant & Branch and bound nodes & Optimality cuts & Time generating cuts (sec.)\\
\midrule
 1 &  0.12  & 6.22 & 0.02 \\
 2 &  446.38  & 254.91 & 95.49 \\
 3 &  3.66 & 4.17 & 355.79 \\
\bottomrule
\end{tabular}
\end{table}

Comparing \Cref{tbl:comp_bench} and \Cref{tbl:perf_variant} we observe that for all three variants the method spent the vast majority of its time generating cuts.
The method was often able to converge at the root node, and generate few cuts, when solving instances of the first variant.
However, the more complicated second stage subproblem associated with the second variant required it to perform a more in-depth search of the branch-and-bound tree as well as generate far more cuts during that search.
Comparing variants two and three we observe that the method did not need to generate many optimality cuts to converge. However, the time required to generate that smaller set of cuts was much greater. 

Given that the proposed method solved instances of the first variant nearly instantaneously, we next report on the time it required to solve the other two variants in more detail. We first note that the method was able to solve the majority (73.12\%) of instances of the second variant within 50 seconds. While instances of the third variant took more time, the method was able to solve 81.23\% within ten minutes. On the other hand, the method needed more than 1,000 seconds to solve 9.14\% of the instances of the third variant.
 
We next seek to understand what instance parameters impact the solution time of the proposed method when applied to instances of variant two or three.  To that effect, we first consider the average solve time, averaged over instances that share the same number of distributions ($\vert \mtc{D} \vert$, \Cref{tbl:solve_num_d_1}). 
We see that while the solution time for instances of variant two increases as the number of distributions increases, the same trend is not present for instances of variant three. Given that few branch and bound nodes were explored when solving instances of variant three  (\Cref{tbl:perf_variant}), this is likely due to few master problem solutions needing to be evaluated before a provably optimal solution was found.

\begin{table}[htp]
\centering
\small
\caption{Average time to solve by $\vert \mtc{D} \vert$.}
\label{tbl:solve_num_d_1}
\begin{subtable}{.8\linewidth}
\caption{$\vert \mtc{D} \vert = 4-64$.}
\begin{tabular}{|c|c|c|c|c|c|c|c|c|}
\hline
 & \multicolumn{8}{c|}{$\vert \mtc{D} \vert$} \\
Variant & 4 & 8 & 9 & 16 & 25 & 27 & 32 & 64 \\
\hline
2 & 2.45 & 4.02 & 4.11 & 7.21 & 9.71 & 10.34 & 15.43 & 23.11 \\
3 & 267.24 & 342.81 & 507.83 & 410.81 & 540.42 & 325.39 & 370.93 & 340.16 \\
\hline
\end{tabular}
\end{subtable}
\begin{subtable}{.8\linewidth}
\caption{$\vert \mtc{D} \vert = 81-3125$.}
\begin{tabular}{|c|c|c|c|c|c|c|c|}
\hline
 & \multicolumn{7}{c|}{$\vert \mtc{D} \vert$} \\
Variant &  81 & 125 & 243 & 256 & 625 & 1024 & 3125 \\
\hline
2 &  31.20 & 41.59 & 95.31 & 93.97 & 272.15 & 589.66 & 672.06 \\
3 &  277.60 & 303.63 & 380.61 & 292.54 & 272.31 & 341.59 & 308.33 \\
\hline
\end{tabular}
\end{subtable}
\end{table}

To further understand the behavior of the method when applied to variant two we present in \Cref{tbl:num_cuts_by_d} the average number of optimality cuts generated when solving instances involving the same number of distributions, $\vert \mtc{D} \vert.$
We see that the method generates more cuts for instances that involve more distributions. However, the method regularly generates less than two optimality cuts per distribution. This suggests that the method is often able to generate a single optimality cut that provides sufficient evidence that first stage variable values from a given subset $\mtc{X}_d$ will lead to worse objective values.

This pattern appears to be broken for the instances with $3125$ distributions.
Recall that the results in \Cref{tbl:num_cuts_by_d} are obtained using instances of variant two.
In \Cref{tbl:comp_bench} it is evident $3.75\%$ of these instances were not solved to optimality.
These are precisely some of the instances with $3125$ distributions.
For these instances the method did not manage to add all the cuts necessary for convergence within the time limit.

\begin{table}[htp]
\centering
\small
\caption{Average number of optimality cuts in variant two by $\vert \mtc{D} \vert$.}
\label{tbl:num_cuts_by_d}
\begin{subtable}{.8\linewidth}
\caption{$\vert \mtc{D} \vert = 4-64$}
\begin{tabular}{|c|c|c|c|c|c|c|c|}
\hline
 \multicolumn{8}{|c|}{$\vert \mtc{D} \vert$} \\
  4 & 8 & 9 & 16 & 25 & 27 & 32 & 64 \\
\hline
 6.74 & 11.52 & 11.82 & 20.98 & 30.40 & 32.00 & 41.16 & 72.44 \\
\hline
\end{tabular}
\end{subtable}
\begin{subtable}{.8\linewidth}
\caption{$\vert \mtc{D} \vert = 81-3125$}
\begin{tabular}{|c|c|c|c|c|c|c|}
\hline
  \multicolumn{7}{|c|}{$\vert \mtc{D} \vert$} \\
  81 & 125 & 243 & 256 & 625 & 1024 & 3125 \\
\hline
89.14 & 134.06 & 254.08 & 267.98 & 636.74 & 1039.58 & 2060.34 \\
\hline
\end{tabular}
\end{subtable}
\end{table}

Next, we report in \Cref{tbl:solve_num_s} the performance of the method on instances with the same number of scenarios per distribution ($S$).
\begin{table}[htp]
\centering
\small
\caption{Average time to solve by $S$.}
\label{tbl:solve_num_s}
\begin{tabular}{|c|c|c|c|c|c|c|c|c|c|c|}
\hline
 & \multicolumn{10}{c|}{$S$} \\
Variant &  5 & 10 & 15 & 20 & 25 & 30 & 35 & 40 & 45 & 50 \\
\hline
2 &  8.03 & 31.00 & 76.82 & 113.38 & 55.11 & 84.47 & 100.66 & 132.75 & 163.32 & 203.38 \\
3 &  8.52 & 39.54 & 84.94 & 142.76 & 197.78 & 292.01 & 463.82 & 608.11 & 879.34 & 841.32 \\
\hline
\end{tabular}
\end{table}
We see that for both variants the solution time increases as the number of scenarios per distribution increases.

We next turn our attention to the impact of the distribution-independent inequalities presented in \Cref{subsec:dist_ind_opt}. To do so, we executed the proposed L-Shaped method on variants two and three again, albeit with the additional valid inequalities \eqref{eq:ext:vi2} that are based on \Cref{prop:min_ev}. We note that given that the subproblems are maximization problems, and concave, the inequality was derived by replacing $\min_{d\in\mtc{D}}\Ex_{\pr{P}_d}[\xi]$ with $\max_{d\in\mtc{D}}\Ex_{\pr{P}_d}[\xi]$. We further note that for variant three \Cref{prop:min_ev} was applied to the LP relaxation of the integer second stage problem. Also, for both variants, distribution-independent inequalities were added only at the root node of the branch-and-bound tree.

We first note that the use of the distribution-independent inequalities did not lead to a significant improvement in the performance of the method when applied to instances of variant two.
Considering variant three, we recall that the method executed without valid inequalities \eqref{eq:ext:vi2} solved every instance, see \Cref{tbl:comp_bench}.
We observe that the method exhibited the same performance with the use of these additional inequalities.
Further, we report in \Cref{tab:max_ev_cuts} the average time the method required with and without these inequalities, for instances with large numbers of distributions.
We see that, for instances involving large number of distributions, the use of these inequalities can significantly decrease the time required by the method to converge to an optimal solution.

\begin{table}[htp]

\caption{
Average time for \textbf{LS} to solve the largest instances of variant three with and without valid inequality \eqref{eq:ext:vi2}.
}\label{tab:max_ev_cuts}
  \centering  
    \begin{tabular}{rrrr}
    \toprule
  $\lvert\mtc{D}\rvert$ & With Dist-ind. VI (sec.) &  Without Dist-ind. VI (sec.) & $\Delta$ (\%) \\
\midrule
 243 & 346.03  & 380.61 & -9.99 \\
 256  & 255.06  & 292.53 & -14.68 \\
 625  & 241.21  & 272.30 & -12.89 \\
 1024  & 309.43  & 341.58 & -10.39 \\
 3125  & 263.13  & 308.32 & -17.17 \\
\bottomrule
\end{tabular}
\end{table}

\section{Conclusions and future work}
\label{sec:conclusion_future}

In this paper, we considered two-stage stochastic programs in which uncertainty is endogeneous.
Namely, the probablity distributions of the data of the second stage problem depend on first stage decisions.
We proposed an adaptation of the L-Shaped method for solving two-stage stochastic programs to this class of problem.
Critical to this adaptation is the concept of \textit{distribution-specific} cuts.
We presented a mechanism for adding such cuts to a master problem solved in the context of executing an L-Shaped method.
We also demonstrated how such cuts can be derived for two important classes of sub-problem.
The first is sub-problems that can be formulated as a linear program.
The second is those that can be formulated as a mixed integer linear program.
We also proposed optimality cuts that are distribution-independent. 
We adapted the general method to three variants of a production planning problem under endogenous and exogenous uncertainty.
With an extensive computational study involving over $2,000$ instances we demonstrated the superior computational performance of the proposed adaptation of the L-Shaped method. 

We see multiple avenues for future work. First, we recall that the validity of the cuts presented for the case of second stage subproblems that are mixed integer linear programs requires an additional assumption regarding the domains of first stage decision variables. Thus, deriving cuts that do not require this assumption would leave a method that is applicable to a broad class of problems. The mechanism we provide for creating distribution-specific cuts can be easily adapted to cuts applicable to different types of first stage variables. Second, the method involves generating distribution-specific inequalities. While we also propose inequalities that are distribution-independent, one can envision enhancing the method with techniques that adapt the inequalities generated for one distribution to be valid for another.  Third, there has been a tremendous amount of research recently on techniques for speeding up the L-Shaped method when solving stochastic programs in which uncertainty is exogenous. Adapting such techniques to the proposed method in this paper is a promising line of research. One example of such a technique is a multi-cut version of the proposed method in which cuts are generated for each distribution and/or each of its scenarios.  Fourth, there are many practical applications in which uncertainty is endogenous. Thus, another line of future work is to adapt the proposed method to problems other than the production planning problem considered in this paper.  Tailoring the method to specific problems may enable further enhancements as problem structure may be exploited. Fifth, preprocessing techniques for identifying when a partition of the feasible region can not contain an optimal solution will likely enable the method to converge much more quickly. The avenues just outlined are primarily methodological. We also believe there is value in extending the well-known concepts of \textit{Value of the Stochastic Solution} and \textit{Expected Value of Perfect Information} to problems with decision-dependent uncertainty.

\bibliographystyle{abbrv}
\bibliography{preprint.bib}
\section{Acknowledgments}
This research was partly supported by the \textit{Novo Nordisk Fonden} grant NNF24OC0089770.

\begin{appendices}
\section{Example reformulation with binary variables}\label{sec:app:example}

Let $\mtc{X}=\{0,1\}^{n_1}$ with $n_1=5$.
Let $n^{(1)}=2$, $n^{(2)}=4$, $n^{(3)}=5$, so that each $x$ vector is split into $T=3$ segments.
Assume there are exactly two exhaustive conditions, say $k_1$ and $k_2$, on segment $t$ given by, respectively, 
$\sum_{i=n^{(t-1)+1}}^{n^{(t)}}x_i<1$ and $\sum_{i=n^{(t-1)+1}}^{n^{(t)}}x_i\geq 1$.
Hence, $\mtc{K}_t=\{k_1,k_2\}$ for $t=1,\ldots,T$.
We obtain $\lvert\mtc{D}\rvert=\Pi_{t=1}^3\lvert\mtc{K}_t\rvert=8$.
Let the function $k(t,d)$ be specified as follows: 
$k(1,d_1)=k_1$, $k(2,d_1)=k_1$, $k(3,d_1)=k_1$,
$k(1,d_2)=k_1$, $k(2,d_2)=k_1$, $k(3,d_2)=k_2$,
$k(1,d_3)=k_1$, $k(2,d_3)=k_2$, $k(3,d_3)=k_1$,
$k(1,d_4)=k_1$, $k(2,d_4)=k_2$, $k(3,d_4)=k_2$,
$k(1,d_5)=k_2$, $k(2,d_5)=k_1$, $k(3,d_5)=k_1$,
$k(1,d_6)=k_2$, $k(2,d_6)=k_1$, $k(3,d_6)=k_2$,
$k(1,d_7)=k_2$, $k(2,d_7)=k_2$, $k(3,d_7)=k_1$,
$k(1,d_8)=k_2$, $k(2,d_8)=k_2$, $k(3,d_8)=k_2$.
The sets $\mtc{X}_d$ are then specified as follows:
$\mtc{X}_{d_1}=\{(0,0,0,0,0)\}$, $\mtc{X}_{d_2}=\{(0,0,0,0,1)\}$
$$\mtc{X}_{d_3}=\{(0,0,0,1,0),(0,0,1,0,0),(0,0,1,1,0)\}$$
$$\mtc{X}_{d_4}=\{(0,0,0,1,1),(0,0,1,0,1),(0,0,1,1,1)\}$$
$$\mtc{X}_{d_5}=\{(1,0,0,0,0),(0,1,0,0,0),(1,1,0,0,0)\}$$
$$\mtc{X}_{d_6}=\{(1,0,0,0,1),(0,1,0,0,1),(1,1,0,0,1)\}$$
$$\mtc{X}_{d_7}=\{x\in\{0,1\}^5\vert \sum_{i=1}^{2}x_i\geq 1, \sum_{i=3}^{4}x_i\geq 1, x_5 = 0\}$$
$$\mtc{X}_{d_8}=\{x\in\{0,1\}^5\vert \sum_{i=1}^{2}x_i\geq 1, \sum_{i=3}^{4}x_i\geq 1, x_5 = 1\}$$
The reformulation proposed in \Cref{sec:benders:mpformulations} entails creating six binary variables $v_{tk}$ instead of eight binary variables $\delta_d$.

\section{Extensive linearized form of problem variant $1$}\label{app:lin}

The extensive linearized form of \eqref{eq:ppp1} is as follows.
\begin{subequations}
\begin{align}
\max~&-\sum_{f\in\mathcal{F}}C_{f}x_{f}+\sum_{d\in\mathcal{D}}\sum_{s\in\mathcal{S}_{d}}\pi_{sd}\bigg(P \mu_{sd}+O\rho_{sd}\bigg)\\
\text{s.t. }~&   \sum_{l\in\mathcal{L}_{f}}y_{fl} =1&\forall f\in\mathcal{F}  \\
&   \sum_{l\in\mathcal{L}_{f}}L_{fl}y_{fl} \leq x_{f}\leq \sum_{l\in\mathcal{L}_{f}}U_{fl}y_{fl}&\forall f\in\mathcal{F}\\ 
  & \sum_{d\in\mathcal{D}}\delta_{d} = 1 & , \\
 &\sum_{f\in\mathcal{F}}y_{f,l(f,d)}\geq |\mathcal{F}|\delta_{d} &\forall  d\in \mathcal{D}\\
&w_{sd} + o_{sd}= \sum_{f\in\mathcal{F}}Y_{fsd}x_{f} & \forall  d\in\mathcal{D}, s\in\mathcal{S}_{d}, \\
 &w_{sd}\leq D_{sd}&\forall d\in\mathcal{D}, s\in\mathcal{S}_{d}\\
 & x_{f} \geq 0  &\forall f\in\mathcal{F}, \\ 
 & y_{fl} \in \{0,1\} &\forall f\in\mathcal{F}\, l \in L_{pf}, \\
 & \delta_{d} \in \{0,1\} &\forall d \in \mathcal{D}, \\
 & w_{sd} \geq 0, o_{sd} \geq 0 & \forall d \in \mathcal{D}, s \in \mathcal{S}_{d}. \\
&\mu_{sd}\leq w_{sd}&\forall  d\in\mathcal{D}, s\in\mathcal{S}_{d}\\   
 &\mu_{sd}\leq D_{sd}\delta_{d}&\forall  d\in\mathcal{D}, s\in\mathcal{S}_{d}\\
 &\mu_{sd}\geq w_{sd} - D_{sd}(1-\delta_{d})&\forall d\in\mathcal{D}, s\in\mathcal{S}_{d}\\
 &\rho_{sd}\leq o_{sd}&\forall  d\in\mathcal{D}, s\in\mathcal{S}_{d}\\   
 &\rho_{sd}\leq N_{sd}\delta_{d}&\forall  d\in\mathcal{D}, s\in\mathcal{S}_{d}\\
&\rho_{sd}\geq o_{sd} - N_{sd}(1-\delta_{d})&\forall d\in\mathcal{D}, s\in\mathcal{S}_{d}.
\end{align}
\end{subequations}
Here, variables $\mu_{sd}$ and $\rho_{sd}$ linearize the products $\delta_{sd}w_{sd}$ and $\delta_{sd}o_{sd}$, respectively.
Constants $M_{sd}$ and $N_{sd}$ represent upper bounds on sales and oversupply, respectively. 

\end{appendices}

\end{document}